\documentclass[reqno]{amsart}
\usepackage{amssymb, amsmath}
\usepackage{textcomp}
\usepackage[all]{xy}
\usepackage{enumerate}
\usepackage{fouridx}

\newdir^{ (}{{}*!/-8pt/@^{(}}

\def\PPP{\mathcal{P}}

\def\AA{\mathbb{A}}

\def\ZZ{\mathbb{Z}}
\def\PP{\mathbb{P}}
\def\GG{\mathbb{G}}
\def\P1{\PP^1}
\def\AAA{\mathbb{A}}
\def\A2{\AAA^2}

\def\Aut{\mathrm{Aut}}
\def\GL{\mathrm{GL}}
\def\SL{\mathrm{SL}}
\def\SO{\mathrm{SO}}
\def\PGL{\mathrm{PGL}}

\def\Gal{\mathrm{Gal}}

\def\Spec{\mathrm{Spec}}

\def\Br{\mathrm{Br}}

\def\ol{\overline}

\newtheorem{theorem}{Theorem}[section]
\newtheorem{lemma}[theorem]{Lemma}
\newtheorem{prop}[theorem]{Proposition} 
\newtheorem{cor}[theorem]{Corollary}

\theoremstyle{definition}
\newtheorem{defin}[theorem]{Definition}
\newtheorem{remark}[theorem]{Remark}
\newtheorem{example}[theorem]{Example}
\newtheorem{constr}[theorem]{Construction}

\begin{document}
\title[Strongly incompressible curves]{Strongly incompressible curves}
\author{Mario Garcia-Armas}
\address{Department of Mathematics, University of British Columbia, Vancouver, \newline \indent %
BC V6T 1Z2, Canada} 
\email{marioga@math.ubc.ca}
\thanks{The author is partially supported by a Four Year Fellowship, University of British Columbia.}

\begin{abstract}
Let $G$ be a finite group. A faithful $G$-variety $X$ is called strongly incompressible if every dominant $G$-equivariant rational map of $X$ onto another faithful $G$-variety $Y$ is birational. We settle the problem of existence of strongly incompressible $G$-curves for any finite group $G$ and any base field $k$.
\end{abstract}

\maketitle

\section{Introduction}

Let $G$ be an algebraic group. A {\em $G$-compression} of a generically free $G$-variety $X$ is a dominant $G$-equivariant rational map $X \dashrightarrow Y$, where $Y$ is also generically free. We say that $X$ is \emph{strongly incompressible} if every $G$-compression of $X$ is birational. This concept was introduced by Z. Reichstein in \cite[\S 2]{Re04}, where the author asks for a classification of strongly incompressible $G$-varieties (see also \cite[\S 7.1]{Re10}).

A related problem arises when we only consider self rational maps. More precisely, given a generically free $G$-variety $X$, is every dominant $G$-equivariant rational map  $X \dashrightarrow X$ a birational isomorphism? Even when $G$ is trivial, this appears to be an interesting problem in many contexts. In \cite{Ch10}, X. Chen proves that every dominant self rational map of a very general projective $K3$ surface of genus $g \geq 2$ is birational (see \cite{Ch12} for generalizations). 

If a finite group $G$ does not act faithfully on any curve of genus $\leq 1$, then there exist strongly incompressible complex $G$-curves (see \cite[ Example 6]{Re04}). N. Fakhruddin and R. Pardini have independently found examples of strongly incompressible complex $G$-surfaces for certain finite groups $G$. To our best knowledge, no examples of strongly incompressible varieties are known in higher dimensions.

If the base field $k$ has characteristic $p>0$, there exist no strongly incompressible $G$-varieties for any finite group $G$. We sketch a proof of this fact. Let $X$ be a faithful $G$-variety and let $F_{X/\Spec(k)} \colon X \to X^{(p)}$ be the relative Frobenius morphism associated to $X$ (see \cite[\S 3.2.4]{Liu02} for details). By functoriality, we may endow $X^{(p)}$ with an action of the group $G^{(p)}$, which is canonically isomorphic to $G$ (recall that $G$ is a finite constant group). This action makes $F_{X/\Spec(k)}$ into a dominant $G$-equivariant morphism, which has degree $p^{\dim(X)}$ by \cite[Cor. 2.27]{Liu02}. To complete the proof, we must show that the $G$-action on $X^{(p)}$ is faithful. If $N$ is the kernel of such action, we must have $k(X^{(p)}) \subset k(X)^N \subset k(X)$, where the inclusion $k(X^{(p)}) \subset k(X)$ is a purely inseparable extension induced by $F_{X/\Spec(k)}$. Thus $k(X)/k(X)^N$ is both Galois and purely inseparable, which implies that $N$ is trivial.

In this paper, we study the question of existence of strongly incompressible $G$-curves for every finite group $G$. We will assume throughout that the base field $k$ has characteristic $0$. We settle the classification problem for $G$-curves raised in \cite{Re04}, by considering finite groups $G$ that can act on a curve of genus $\leq 1$. In Section \ref{Sect_Incompressibility_Curves}, we show that strongly incompressible $G$-curves do exist if $G$ does not act faithfully on any curve of genus $0$.

\begin{theorem}[see Theorem \ref{Thm_Incompressible_No_Action_Genus0}]
Suppose that $G$ cannot act faithfully on a curve of genus $0$ via $k$-morphisms. Then there exists a strongly incompressible $G$-curve defined over $k$. 
\end{theorem}

For finite groups $G$ that can act faithfully on a curve of genus $0$ over $k$ (recall that these are always cyclic, dihedral, or polyhedral groups), the situation is more delicate. In particular, it is important to decide whether a faithful $G$-curve $X$ can be $G$-compressed to $\PP^1$, provided that there exists a faithful $G$-action on the projective line. To this end, we make a small detour in Section \ref{Sect_Equiv_ProjSpaces} and given a projective representation $G \to \PGL(V)$, we construct a cohomological invariant associated to any faithful $G$-variety $X$, which allows us to determine whether $X$ can be mapped $G$-equivariantly to $\PP(V)$. In Section \ref{Sect_Actions_Proj_Line}, we compute the invariant for certain group actions on the projective line.  

We study the existence of strongly incompressible curves for groups that can act faithfully on a curve of genus $0$ in Sections \ref{Sect_Incompressibility_Cyclic_Dihedral} to \ref{Sect_Polyhedral}. Our results are summarized in the following theorem. For a definition of cohomological $2$-dimension of a field $k$, denoted by $\mathrm{cd}_2(k)$, we refer the reader to \cite[I.\S 3]{Se02}. We remark that $k$ has cohomological $2$-dimension $0$ if and only if every algebraic extension of $k$ is quadratically closed (see \cite[Lemma 2]{EW87}).

\begin{theorem}
Let $n \geq 2$ be an integer and let $\omega_n$ be a primitive $n$-th root of $1$.
\begin{enumerate}[\upshape(a)]
\item (Thm. \ref{Thm_Incompressible_No_Action_Genus0}, Prop. \ref{Prop_Odd_Dihedral_Incompressibility}) Let $G$ be either $\ZZ/n\ZZ$ or $D_{2n}$, where $n$ is odd. Then there exist strongly incompressible $G$-curves if and only if $\omega_n+\omega_n^{-1} \not \in k$.
\item (Thm. \ref{Thm_Incompressible_No_Action_Genus0}, Prop. \ref{Prop_Cyclic_Versal_Incompressibility} and \ref{Prop_Even_Cyclic_Not_Versal_Incompressibility}) Suppose that $n$ is even. Then there exist strongly incompressible $\ZZ/n\ZZ$-curves if and only if $\omega_n \not\in k$.
\item (Prop. \ref{Prop_Klein_Incompressibility}) There exist strongly incompressible $(\ZZ/2\ZZ)^2$-curves if and only if $\mathrm{cd}_2(k) > 0$.
\item (Thm. \ref{Thm_Incompressible_No_Action_Genus0}, Prop. \ref{Prop_Even_Dihedral_Incompressibility}) Suppose that $n \geq 4$ is even. Then there exist strongly incompressible $D_{2n}$-curves if and only if either $\omega_n+\omega_n^{-1} \not\in k$, or $\mathrm{cd}_2(k) > 0$. 
\item (Prop. \ref{Prop_Incompressibility_Polyhedral}) Let $G$ be a polyhedral group, i.e., $G = A_4$, $S_4$, or $A_5$. Then there exist strongly incompressible $G$-curves if and only if $\mathrm{cd}_2(k) > 0$.
\end{enumerate}
\end{theorem}

In particular, we note the following corollary of the above results, which answers the strong incompressibility problem for curves over an algebraically closed field, as posed in \cite{Re04}.

\begin{cor}
Let $G$ be a finite group and let $k$ be an algebraically closed base field. Then there exists a strongly incompressible $G$-curve if and only if $G$ does not act faithfully on $\PP^1$, i.e., $G$ is not cyclic, dihedral, $A_4$, $S_4$ or $A_5$.
\end{cor}

\bigskip
\noindent{\bf Acknowledgements.}
I am extremely grateful to Z. Reichstein for introducing me to the subject. This work would not have been possible without his helpful comments and suggestions. I also thank Bob Guralnick for useful conversations.

\section{Notation and preliminaries} 
Let $k$ denote a base field of characteristic $0$ and let $\ol{k}$ be its algebraic closure. A {\em $k$-variety} $X$ is a geometrically reduced scheme of finite type over $k$ (not necessarily irreducible). The word ``curve'' is reserved for a geometrically irreducible smooth projective $1$-dimensional variety. A {\em point} of a variety means a geometric point, unless stated otherwise.

As usual, a rational map $X \dashrightarrow Y$ of $k$-varieties is an equivalence class of $k$-morphisms $U \to Y$, where $U$ is a dense open subset of $X$. We denote the algebra of rational functions of $X$ by $k(X)$. In general, $k(X)$ is the direct sum of the function fields of the irreducible components of $X$. 

An {\em algebraic group $G$ over $k$} is a smooth affine group scheme of finite type over $k$. We say that $X$ is a {\em $G$-variety} if $G$ acts morphically on $X$. The inclusion $k(X)^G \hookrightarrow k(X)$ induces a {\em rational quotient map} $\pi_X\colon X \dashrightarrow W$, where $k(W) = k(X)^G$ (see \cite[\S 2.3]{Re00}). The variety $W$ is denoted by $X/G$ and is unique up to birational isomorphism.
If $N$ is a normal subgroup of $G$, there exists a model of $X/N$ with a regular action of $G/N$ (see \cite[Remark 2.6]{Re00}). It is uniquely defined up to $G/N$-equivariant birational isomorphism. A rational map $X \dashrightarrow Y$ of $G$-varieties gives rise to a $G/N$-equivariant rational map $\ol{f}\colon X/N \dashrightarrow Y/N$ such that $\ol{f} \circ \pi'_X = \pi'_Y \circ f$, where $\pi'_X\colon X \dashrightarrow X/N$ and $\pi'_Y\colon Y \dashrightarrow Y/N$ are the rational quotient maps.

A $G$-action on $X$ is said to be {\em generically free} if there exists a dense $G$-invariant open subset of $X$ with trivial scheme-theoretic stabilizers. (In particular, a faithful action of a finite group is generically free.) A {\em $G$-compression} is a $G$-equivariant dominant rational map $X \dashrightarrow Y$, where $X$ and $Y$ are generically free $G$-varieties. A generically free $G$-variety $X$ contains a dense $G$-invariant open subset $U$ which is the total space of a $G$-torsor $\pi_U\colon U \to U/G$ (see \cite[Thm. 4.7]{BF03}). We say that $X$ is {\em primitive} if $G$ transitively permutes the irreducible components of $X$ (equivalently, if $X/G$ is irreducible). Under this condition, the fiber at the generic point of $U/G$ is a $G$-torsor $T \to \Spec(K)$, where $K \cong k(X)^G$. The class of this torsor in $H^1(K,G)$ will be denoted by $[X]$. Conversely, given a finitely generated field extension $K$ of $k$, any class in $H^1(K,G)$ determines a generically free primitive $G$-variety $X$ endowed with a $k$-isomorphism $k(X)^G \cong K$, uniquely up to $G$-equivariant birational isomorphism. In what follows, we assume all $G$-varieties to be primitive, unless stated otherwise.

Given a central simple algebra $A$, we will denote its Brauer class by $[A]$. As usual, the symbol $(a,b)_2$ denotes the quaternion algebra with basis $1, i, j, ij$, subject to the relations $i^2=a, j^2=b$ and $ij+ji=0$. The following simple observation will be used repeatedly in the sequel.

\begin{lemma} \label{Lemma_Square_Quaternions}
Let $k(x)$ be a rational function field over $k$, and suppose that the quaternion algebra $(f(x),g(x))_2$ is split over $k(x)$, where $f, g \in k[x]$ are separable. Then $f(\alpha)$ is a square in $k(\alpha)$ for any root $\alpha \in \ol{k}$ of $g$.  
\end{lemma}

\begin{proof}
Since the quaternion algebra $(f(x),g(x))_2$ is split, there exist coprime polynomials $p ,q, r \in k[x]$ such that the polynomial identity
$$
f(x) p(x)^2 + g(x) q(x)^2 = r(x)^2
$$
holds. Substituting $\alpha$ in the above identity implies that $f(\alpha) p(\alpha)^2 = r(\alpha)^2$. Note that $p(\alpha)=0$ implies $r(\alpha) = 0$. Conversely, suppose that $r(\alpha) = 0$. Then $\alpha$ is a root of $f(x) p(x)^2$ of multiplicity at least $2$, which implies that $p(\alpha) = 0$ since $f$ is separable. It follows that $r(\alpha)=0$ if and only if $p(\alpha) = 0$.

Assume for the sake of contradiction that $p(\alpha) = r(\alpha) = 0$. Then it follows that $\alpha$ is a root of $g(x) q(x)^2$ of multiplicity at least $2$. Since $g$ is separable, we obtain that $q(\alpha) = 0$. Hence $\alpha$ is a common root of $p, q, r$, which is impossible since they are relatively prime. This contradiction shows that $p(\alpha) r(\alpha) \neq 0$ and therefore $f(\alpha) = r(\alpha)^2 p(\alpha)^{-2} \in k(\alpha)^{\times 2}$.
\end{proof}

\section{Strong Incompressibility of Curves} \label{Sect_Incompressibility_Curves}

Let $G$ be a finite group. Recall that a faithful $G$-variety $X$ is said to be strongly incompressible if any $G$-compression $X \dashrightarrow Y$ onto a faithful $G$-variety $Y$ is birational. We are interested in the study of strong incompressibility of $G$-curves. We remark that the existence of strongly incompressible $G$-curves depends not only on the group $G$, but also on the base field $k$.

Note also that $G$-compressions of curves extend naturally to $G$-equivariant surjective finite morphisms, so we will regard $G$-compressions of curves as morphisms in the sequel. The following simple lemma is extremely useful in our analysis.

\begin{lemma}[{cf. \cite[Example 6]{Re04}}] \label{Lemma_NoCompression_Small_Genus}
Suppose that there exists a faithful $G$-curve $X$ that cannot be $G$-compressed to any $G$-curve of genus $\leq 1$. Then there exists a strongly incompressible $G$-curve.
\end{lemma}

\begin{proof}
Consider the set $S$ consisting of faithful $G$-curves $Y$ such that there exists a $G$-compression $X \to Y$. By assumption, the genus $g(Y) \geq 2$ for all $Y \in S$. Select a curve $Y_0 \in S$ having minimal genus. We claim that $Y_0$ is strongly incompressible. Indeed, suppose that we have a $G$-compression $f\colon Y_0 \to Y'$, which implies that $Y' \in S$. In particular, we must have $g(Y') \geq g(Y_0) \geq 2$. However, by Hurwitz Formula (see \cite[Thm 7.4.16]{Liu02}) it also follows that $g(Y_0) \geq g(Y')$, whence equality must hold. This implies that $f$ is birational. 
\end{proof}

The following result will be instrumental in the sequel. It is a special case of \cite[Prop. 8.6]{RY01} (see also \cite[Rem. 9.9]{RY01}), whose proof depends on resolution of singularities. We include an alternative proof because it works over any base field of characteristic $0$, it is more elementary and, in particular, does not rely on resolution of singularities.

\begin{theorem} \label{Theorem_Curve_FixedPoints}
There exists a faithful $G$-curve $X$ defined over $k$ such that every element of $G$ fixes some geometric point of $X$.
\end{theorem}

\begin{proof}
See Appendix \ref{Appendix_Curves}.
\end{proof}

We now recall some facts about the automorphism group of an elliptic curve. 

\begin{lemma} \label{Lemma_Facts_Elliptic}
Let $E$ be an elliptic curve defined over a field $k$.
\begin{enumerate} [\upshape(a)]
\item There exists a split exact sequence
\begin{equation} \label{Eq_Exact_Seq_Elliptic}
\xymatrix{1 \ar[r]& E \ar[r]^-{i} &\Aut(E) \ar[r]^-{\pi}& \Aut_0(E) \ar[r] & 1},
\end{equation}
where $E$ acts on itself by translations and $\Aut_0(E)$ denotes the group of automorphisms of $E$ that preserve the origin.
\item There exists a natural isomorphism $\Aut_0(E) \cong \mu_n$, where 
$$
n = \left\{ \begin{array}{cl}
2, & \textrm{if } j(E) \neq 0, 1728;\\
4, & \textrm{if } j(E) = 1728;\\
6, & \textrm{if } j(E) = 0.
\end{array}\right.
$$

\item If $j(E) = 1728$ (resp. $0$), we have $\Aut_0(E)(k) = \ZZ/4\ZZ$ (resp. $\ZZ/6\ZZ$) if and only if $\omega_4 \in k$ (resp. $\omega_3 \in k$).
\item The translation by $P_0 \in E$ and the automorphism $\alpha \in \Aut_0(E)$ commute if and only if $\alpha(P_0) = P_0$. 
\end{enumerate}
\end{lemma}

\begin{proof}
(a) See, e.g., \cite[\S X.5]{Sil09}. Note that in \cite{Sil09}, $\Aut(E)$ and $\Aut_0(E)$ are denoted by $\textrm{Isom}(E)$ and $\Aut(E)$, respectively.

(b) See \cite[Cor. III.10.2]{Sil09}.

(c) This follows directly from part (c).

(d) Let $\tau_{P_0}$ denote the translation by $P_0$. Then note that $\tau_{P_0}$ and $\alpha$ commute if and only if $\alpha(P) + \alpha(P_0) = \alpha \circ \tau_{P_0} (P) = \tau_{P_0} \circ \alpha (P) = \alpha(P) + P_0$
for all $P \in E$, which implies the desired result. 
\end{proof}

\begin{theorem} \label{Thm_Incompressible_No_Action_Genus0}
Suppose that $G$ cannot act faithfully on a curve of genus $0$ via $k$-morphisms. Then there exists a strongly incompressible $G$-curve. 
\end{theorem}

\begin{proof}
By Lemma \ref{Lemma_NoCompression_Small_Genus}, it suffices to prove that there exists a faithful $G$-curve $X$ that cannot be $G$-compressed to any curve of genus $1$.

Note that G is not isomorphic to $\ZZ/n\ZZ$ for $n = 1, 2, 3, 4, 6$, because these groups act faithfully on $\PP^1$ over $k$ (see, e.g., \cite{Beau10}). By Theorem \ref{Theorem_Curve_FixedPoints}, there exists a faithful $G$-curve $X$ such that every $g \in G$ fixes a geometric point of $X$. For the sake of contradiction, suppose that there exists a $G$-compression $X \to E$, where $E$ is a curve of genus $1$ endowed with a faithful $G$-action. Extending to the algebraic closure, we obtain a $G$-compression $X_{\ol{k}} \to E_{\ol{k}}$. Regard $G$ as a subgroup of $\Aut(E_{\ol{k}})$. By the exact sequence \eqref{Eq_Exact_Seq_Elliptic} and the fact that $G \not\cong \ZZ/n\ZZ$ for $n=1, 2, 3, 4, 6$; we conclude that $G \cap i(E_{\ol{k}}) \neq \emptyset$. Since $i(E_{\ol{k}})$ acts on $E_{\ol{k}}$ by translations, $G \cap i(E_{\ol{k}})$ acts freely on $E_{\ol{k}}$. However, every element of $G$ must fix a point on $E_{\ol{k}}$ by our assumption on $X_{\ol{k}}$. This contradiction shows that $X$ cannot be $G$-compressed to any $G$-curve of genus $1$. 
\end{proof}

In view of the above theorem, it remains to study the existence of strongly incompressible $G$-curves when $G$ can act faithfully on a curve of genus $0$. We will devote Section \ref{Sect_Equiv_ProjSpaces} to the study of equivariant rational maps to projective spaces, and we will use these results to understand compressions onto curves of genus $0$.

\section{Equivariant maps to projective spaces}\label{Sect_Equiv_ProjSpaces}

Let $G$ be an algebraic group defined over a field $k$. A projective representation $\rho\colon G \hookrightarrow \PGL(V)$ gives rise to a $G$-action on $\PP(V)$. We will denote the resulting $G$-variety by $\fourIdx{}{\rho}{}{}{\PP(V)}$. If $\rho$ and $\sigma$ are projective $G$-representations, it is clear that $\fourIdx{}{\rho}{}{}{\PP(V)}$ and $\fourIdx{}{\sigma}{}{}{\PP(V)}$ are $G$-equivariantly isomorphic if and only if $\rho$ and $\sigma$ are conjugate. In what follows, we always assume that the $G$-action on $\fourIdx{}{\rho}{}{}{\PP(V)}$ is generically free.

Consider the commutative diagram whose rows are central exact sequences
\begin{equation}\label{Eq_Exact_seq} 
\xymatrix{1 \ar[r] & \GG_m \ar[r] \ar@{=}[d] & \GL(V) \ar[r]  & \PGL(V) \ar[r] & 1\\
1 \ar[r] & \GG_m \ar[r] & G' \ar@{^{ (}->}[u]_-{\ol{\rho}} \ar[r] & G \ar@{^{ (}->}[u]_-{\rho} \ar[r] & 1}
\end{equation}
where $G'$ is the full preimage of $G$ in $\GL(V)$. Given a field extension $K/k$, we obtain the corresponding diagram in cohomology
\begin{equation} \label{Eq_Exact_seq_cohom}
\xymatrix{ & 1 \ar[r]  & H^1(K,\PGL(V)) \ar[r] & H^2(K,\GG_m)\\
1 \ar[r] & H^1(K,G') \ar[u]_-{\ol{\rho}_*} \ar[r]^-{\varphi} & H^1(K,G) \ar[u]_-{\rho_*} \ar[r]^-{\Delta_{\rho}} & H^2(K,\GG_m) \ar@{=}[u]}
\end{equation}
(Note that $H^1(K,\GG_m)$ and $H^1(K,\GL(V))$ are trivial by Hilbert's Theorem 90.) This construction defines a cohomological invariant $\Delta_\rho \colon H^1(K,G) \to H^2(K,\GG_m) = \Br(K)$. If $X$ is a generically free primitive $G$-variety and $L = k(X)^G$, we denote the Brauer class associated to $[X] \in H^1(L,G)$ by $\Delta_{\rho}(X)$. Note that $\Delta_\rho(X)$ is trivial if and only if $[X]$ lifts to a $G'$-torsor $[X'] \in H^1(L,G')$.

\begin{constr}\label{Construction_Cone}
Let $Y$ be a primitive closed $G$-subvariety of $\fourIdx{}{\rho}{}{}{\PP(V)}$. Endow $V$ with a linear $G'$-action via $\ol{\rho}$ and define $\widetilde{Y} \subset V$ to be the affine cone over $Y$ {\em with the origin removed}. It is not hard to see that $\widetilde{Y}$ is a primitive $G'$-variety. Moreover, it is well known that $\widetilde{Y}$ is a $\GG_m$-torsor and $Y$ is isomorphic to the {\em geometric} quotient $\widetilde{Y}/\GG_m$. Note also that the group $G'/\GG_m \cong G$ acts naturally on $\widetilde{Y}/\GG_m$, in such a way that the above isomorphism is $G$-equivariant.
\end{constr}

\begin{lemma} \label{Lemma_Class_PV}
Let $Y$ be a generically free primitive closed $G$-subvariety of $\fourIdx{}{\rho}{}{}{\PP(V)}$. Then $\Delta_{\rho}(Y)$ is trivial.
\end{lemma}

\begin{proof}
We need to show that $[Y]$ is in the image of the map $\varphi\colon H^1(K,G') \to H^1(K,G)$, where $K = k(Y)^G$. Let $\widetilde{Y}$ be as in Construction \ref{Construction_Cone}. If $x \in Y$ has trivial stabilizer in $G$, then any lift $\widetilde{x} \in \widetilde{Y}$ of $x$ has trivial stabilizer in $G'$. It follows that $\widetilde{Y}$ is a generically free primitive $G'$-variety and clearly $\varphi([\widetilde{Y}]) = [\widetilde{Y}/\GG_m] = [Y]$.
\end{proof}

\begin{prop} \label{Prop_Maps_PV}
Let $G$ be a finite group, let $\rho\colon G \hookrightarrow \PGL(V)$ be a projective representation and let $X$ be a faithful primitive $G$-variety.
\begin{enumerate}[\upshape(a)]
\item Suppose that there exists a $G$-equivariant rational map $f\colon X \dashrightarrow \fourIdx{}{\rho}{}{}{\PP(V)}$. Then $\Delta_{\rho}(X)$ is trivial.
\item Conversely, suppose that $\Delta_{\rho}(X)$ is trivial. Then, given any $G$-invariant open subset $U \subset \fourIdx{}{\rho}{}{}{\PP(V)}$, there exists a $G$-equivariant rational map $X \dashrightarrow U$. 
\end{enumerate}
\end{prop}

\begin{proof}
(a) We write $Y = \ol{f(X)}$, $K = k(Y)^G$ and $L = k(X)^G$. We separate the proof into two cases.

Case 1: Suppose that $Y$ is a faithful $G$-variety. This case follows from the fact that $\Delta_\rho$ is a cohomological invariant. The $G$-compression $f\colon X \dashrightarrow Y$ naturally induces a $k$-field homomorphism $i\colon K \hookrightarrow L$ and we have a commutative diagram
\begin{equation} 
\xymatrix{H^1(K,G) \ar[r]^-{\Delta^K_{\rho}} \ar[d]_-{i_*} & H^2(K,\GG_m) \ar[d] \\
H^1(L,G) \ar[r]^-{\Delta^L_{\rho}}  & H^2(L,\GG_m)
}
\end{equation}
It is well known that in the above situation, we must have $i_*([Y]) = [X]$.
By Lemma \ref{Lemma_Class_PV}, we have $\Delta^K_{\rho}(Y) = 1$. The commutativity of the above diagram then implies that $\Delta^L_{\rho}(X) = 1$.

Case 2: Suppose that the $G$-action on $Y$ has a kernel $H$. Let $\widetilde{Y}$ be as in Construction \ref{Construction_Cone}, and let $H'$ be the kernel of the $G'$-action on $\widetilde{Y}$. We claim that $\pi^{-1}(H)$ splits as $\GG_m \times H'$, where $\pi\colon G' \to G$ is the natural projection. Indeed, let $h \in H$ and let $h' \in \pi^{-1}(h)$ be any lift. Note that for any $\widetilde{y} \in \widetilde{Y}$, there exists $\lambda_{h'}(\widetilde{y}) \in \GG_m$ such that $h' \cdot \widetilde{y} = \lambda_{h'}(\widetilde{y}) \widetilde{y}$. Moreover, the map $\widetilde{y} \mapsto \lambda_{h'}(\widetilde{y})$ is the composition
\begin{equation*}
\widetilde{Y} \to G' \times \widetilde{Y} \to \widetilde{Y} \times \widetilde{Y} \to \GG_m
\end{equation*}
sending $\widetilde{y} \mapsto (h',\widetilde{y}) \mapsto (\widetilde{y},h'\cdot \widetilde{y}) \mapsto \lambda_{h'}(\widetilde{y})$, where the last map is the ``difference'' morphism of the $\GG_m$-torsor $\widetilde{Y}$. We have thus constructed a morphism $\widetilde{Y} \to \GG_m$, which is clearly invariant under the $\GG_m$-action of $\widetilde{Y}$. By properties of geometric quotients, the above morphism descends to a morphism $Y \to \GG_m$. Since $Y$ is projective, this morphism should be constant. From now on, we denote $\lambda_{h'}(\widetilde{y})$ simply by $\lambda_{h'}$. Note that $\lambda_{h'}^{-1} h'$ is the only element in $\pi^{-1}(h)$ that lies in $H'$. It follows that the section $s\colon H \to \pi^{-1}(H)$ given by $h \to \lambda_{h'}^{-1} h'$ is a well-defined homomorphism satisfying $s(H) = H'$. Hence the exact sequence $1 \to \GG_m \to \pi^{-1}(H) \to H \to 1$ splits in the desired way. This finishes the proof of the claim.

It follows that we have a commutative diagram with exact rows
$$
\xymatrix{
1 \ar[r] & \GG_m \ar[r] \ar@{=}[d] & G' \ar[r]^-\pi \ar[d] & G \ar[r] \ar[d] & 1\\
1 \ar[r] & \GG_m \ar[r] & G'/H' \ar[r] & G/H \ar[r] & 1
}
$$
The dominant $G$-equivariant rational map $X \dashrightarrow Y$ induces a $G/H$-compression $X/H \to Y$, which gives rise to a $k$-field homomorphism $i \colon K \hookrightarrow L$. Using the bottom sequence above, we obtain a commutative diagram in cohomology
$$
\xymatrix{H^1(K,G'/H') \ar[r] \ar[d] & H^1(K,G/H) \ar[r] \ar[d]_-{i_*} & H^2(K,\GG_m) \ar[d] \\
H^1(L,G'/H') \ar[r] & H^1(L,G/H) \ar[r]  & H^2(L,\GG_m)
} 
$$
The $G/H$-variety $Y$ represents a class $[Y] \in H^1(K,G/H)$, which maps to $[X/H] \in H^1(L,G/H)$ under $i_*$. It is easy to see that the $G'/H'$-action on $\widetilde{Y}$ is generically free, so it follows that $[Y]$ comes from a class in $H^1(K,G'/H')$ and therefore its image in $H^2(K,\GG_m)$ is trivial. By the commutativity of the above diagram, the image of $[X/H]$ in $H^2(L,\GG_m)$ is also trivial.

To complete the proof of Case 2, note that we have a commutative diagram
$$
\xymatrix{H^1(L,G') \ar[r] \ar[d] & H^1(L,G) \ar[r]^{\Delta_{\rho}} \ar[d] & H^2(L,\GG_m) \ar@{=}[d] \\
H^1(L,G'/H') \ar[r] & H^1(L,G/H) \ar[r]  & H^2(L,\GG_m)
}
$$ 
The image of $[X] \in H^1(L,G)$ under the middle vertical map is precisely $[X/H]$. It thus follows that $\Delta_{\rho}(X)$ is trivial.

(b) By assumption, $[X]$ can be lifted to a class in $H^1(L,G')$, i.e., there exists a generically free primitive $G'$-variety $X'$ such that $X'/\GG_m$ is birationally isomorphic to $X$ as a $G$-variety. Without loss of generality, we may identify $X'/\GG_m$ with $X$. 

We may view $V$ as a generically free linear $G'$-variety and the natural projection $\pi_V\colon V \dashrightarrow \fourIdx{}{\rho}{}{}{\PP(V)}$ as a rational quotient map.
Let $U' = \pi_V^{-1}(U)$, which is clearly a $G'$-invariant open subset of $V$. Note that $V$ is a versal $G'$-variety (see, e.g., \cite[Example 5.4]{GMS03}), whence there exists a $G'$-equivariant rational map $X' \dashrightarrow U'$. Taking quotients by $\GG_m$, we obtain a $G$-equivariant rational map $X = X'/\GG_m \dashrightarrow U'/\GG_m = U$.
\end{proof}

We record the following corollary for future reference. 

\begin{cor} \label{Cor_Maps_P1}
Let $\rho\colon G \hookrightarrow \PGL_2$ be a projective representation of a nontrivial finite group $G$ and let $X$ be a faithful irreducible $G$-variety. Then there exists a $G$-compression $X \dashrightarrow \fourIdx{}{\rho}{}{}{\PP^1}$ if and only if $\Delta_{\rho}(X) = 1$.
\end{cor}

\begin{proof}
The ``only if'' part follows directly from Proposition \ref{Prop_Maps_PV}(a). On the other hand, suppose that $\Delta_{\rho}(X) = 1$. Since $G$ is nontrivial, $\fourIdx{}{\rho}{}{}{\PP^1}$ has a finite number of $G$-fixed points. Therefore, we can find a $G$-invariant open $U \subset \fourIdx{}{\rho}{}{}{\PP^1}$ not containing any $G$-fixed points. By Proposition \ref{Prop_Maps_PV}(b), there exists a $G$-equivariant rational map $f\colon X \dashrightarrow \fourIdx{}{\rho}{}{}{\PP^1}$ such that $f(X) \subset U$. The closure $\overline{f(X)}$ is a $G$-invariant closed irreducible subset of $\fourIdx{}{\rho}{}{}{\PP^1}$. By construction, it cannot be a fixed point, so it coincides with $\fourIdx{}{\rho}{}{}{\PP^1}$ itself. This proves that $f$ is dominant.
\end{proof}

\section{Some explicit computations}
\label{Sect_Actions_Proj_Line}

In this section, we explicitly compute the invariant introduced in Section \ref{Sect_Equiv_ProjSpaces} for certain actions on the projective line. We will use these results later to study the strong incompressibility of $G$-curves in the case where $G$ acts faithfully on $\PP^1$. In what follows, the class of an element $a \in k^\times$ in $k^\times/k^{\times 2}$ will be denoted by $\ol{a}$.

Recall that the conjugacy classes of embeddings of $(\ZZ/2\ZZ)^2$ into $\PGL_2(k)$ are parametrized by the pairs $(\ol{a},\ol{b}) \in (k^\times/k^{\times 2})^2$ such that the quaternion algebra $(a,b)_2$ is split (see \cite{Beau10, Ga13}).  We denote the corresponding embedding by $\rho_{(a,b)}$ and fix generators $e_1, e_2$ of $(\ZZ/2\ZZ)^2$. We have the following three cases:
\begin{itemize}
\item Suppose that both $a$ and $b$ are non-squares. Then we have
\begin{equation} \label{Eq_embedding_Klein1}
\rho_{(a,b)}\colon e_1 \mapsto \left(\begin{array}{cc}
\lambda & -a\\
1 & -\lambda
\end{array}\right),\quad
e_2 \mapsto \left(\begin{array}{cc}
0 & a\\
1 & 0
\end{array}\right),
\end{equation}
where $\lambda^2-a \equiv b \mod k^{\times 2}$ (we can find such $\lambda \in k$ because $(a,b)_2$ is split).
\item Suppose that $\ol{a} = \ol{1}$.  Then we have
\begin{equation} \label{Eq_embedding_Klein2}
\rho_{(a,b)}\colon e_1 \mapsto \left(\begin{array}{cc}
0 & b\\
1 & 0
\end{array}\right),\quad
e_2 \mapsto \left(\begin{array}{cc}
-1 & 0\\
0 & 1
\end{array}\right).
\end{equation}
\item Suppose that $\ol{b} = \ol{1}$.  Then we have
\begin{equation} \label{Eq_embedding_Klein3}
\rho_{(a,b)}\colon e_1 \mapsto \left(\begin{array}{cc}
-1 & 0\\
0 & 1
\end{array}\right),\quad
e_2 \mapsto \left(\begin{array}{cc}
0 & a\\
1 & 0
\end{array}\right).
\end{equation}
\end{itemize} 
(If $(\ol{a},\ol{b}) = (\ol{1},\ol{1})$, the last two embeddings are conjugate.) For simplicity, denote the $(\ZZ/2\ZZ)^2$-variety $\fourIdx{}{\rho_{(a,b)}}{}{}{\PP^1}$ by $\fourIdx{}{(a,b)}{}{}{\PP^1}$. Clearly, $\fourIdx{}{(a,b)}{}{}{\PP^1}$ and $\fourIdx{}{(a',b')}{}{}{\PP^1}$ are isomorphic as $(\ZZ/2\ZZ)^2$-varieties if and only if $\ol{a} = \ol{a'}$ and $\ol{b} = \ol{b'}$.

\begin{lemma} \label{Lemma_Cohom_Kleingroup}
Let $\rho_{(a,b)}$ be as above, let $K/k$ be a field extension and identify $H^1(K,(\ZZ/2\ZZ)^2)$ with $(K^\times/K^{\times 2})^2$. Then $\Delta_{\rho_{(a,b)}}(\ol{c},\ol{d}) = [(ac,bd)_2]$ for all $c,d \in K^\times$.
\end{lemma}

\begin{proof}
It suffices to prove that $\rho_{(a,b)*}\colon (K^\times/K^{\times 2})^2 \to H^1(K,\PGL_2)$ maps $(\ol{c},\ol{d})$ to $(ac,bd)_2$. Let $U, V \in \GL_2$ be lifts of $\rho_{(a,b)}(e_1)$, $\rho_{(a,b)}(e_2)$, respectively. Note that $U^2 = b' I$, $V^2 = a' I$ and $UV + VU = 0$, where $\ol{a'} = \ol{a}$ and $\ol{b'} = \ol{b}$. Rescaling the lifts if necessary, we may assume that $a'=a$ and $b'=b$. Let $\mathcal{A}$ be the split quaternion algebra $(b,a)_2$. Note that there is a $k$-algebra isomorphism $\mathcal{A} \cong \mathrm{M}_2$ given by $i \mapsto U$, $j \mapsto V$, which induces isomorphisms $\GL_1(\mathcal{A}) \cong \GL_2$ and $\PGL_1(\mathcal{A}) \cong \PGL_2$. By construction, $\rho_{(a,b)}$ factors as 
$$
\xymatrix{
(\ZZ/2\ZZ)^2 \ar@{^{ (}->}[r]^\varphi & \PGL_1(\mathcal{A}) \ar[r]^{\cong} & \PGL_2,
}
$$ 
where the embedding $\varphi$ is given by $e_1 \mapsto [i]$, $e_2 \mapsto [j]$. We have therefore reduced the problem to showing that $\varphi_*\colon (K^\times/K^{\times 2})^2 \to H^1(K,\PGL_1(\mathcal{A}))$ sends $(\ol{c},\ol{d})$ to $(ac,bd)_2$ for all $c,d \in K^\times$.

We give a proof of this fact by Galois descent. Let $L = K(\sqrt{c}, \sqrt{d})$; then we may view $\varphi_*(\ol{c},\ol{d})$ as an element of $H^1(\Gal(L/K),\PGL_1(\mathcal{A})(L))$. For simplicity, assume that $c$, $d$ and $c d$ are non-squares; the remaining cases are easier and left to the reader. Define generators $\sigma_1, \sigma_2 \in \Gal(L/K)$ such that $\sigma_1$ fixes $\sqrt{d}$ and sends $\sqrt{c}$ to $-\sqrt{c}$, while  $\sigma_2$ fixes $\sqrt{c}$ and sends $\sqrt{d}$ to $-\sqrt{d}$. Note that the $1$-cocycle $v\colon \Gal(L/K) \to \PGL_1(\mathcal{A})(L)$ representing $\varphi_*(\ol{c},\ol{d})$ is given by $\sigma_1 \mapsto [i]$, $\sigma_2 \mapsto [j]$. Then we twist the Galois action on $\gamma = x + y i + z j + t i j \in \mathcal{A} \otimes_{K} L$ by setting
\begin{align*}
\sigma_1 * \gamma & =  v_{\sigma_1} (\sigma_1(\gamma)) = i^{-1} \sigma_1(\gamma) i = \sigma_1(x) + \sigma_1(y) i - \sigma_1(z) j - \sigma_1(t) i j;\\
\sigma_2 * \gamma & = v_{\sigma_2} (\sigma_2(\gamma)) = j^{-1} \sigma(\gamma) j = \sigma_2(x) - \sigma_2(y) i + \sigma_2(z) j - \sigma_2(t) i j.
\end{align*}
It follows that $\gamma$ is invariant under the twisted Galois action if and only if $\gamma = x + y_1 \sqrt{d}\, i + z_1 \sqrt{c}\, j + t_1 \sqrt{c d}\, i j$ for $x, y_1, z_1, t_1 \in K$. This implies that $\varphi_*(\ol{c},\ol{d})$ is generated as a $K$-algebra by $i' = \sqrt{d}\, i$ and $j' = \sqrt{c}\, j$, which satisfy $i'^2 = bd$, $j'^2 = ac$ and $i'j' + j'i' = 0$. Consequently, we obtain that $\varphi_*(\ol{c},\ol{d}) = (bd,ac)_2 \cong (ac,bd)_2$.
\end{proof}

Recall now that the group $\ZZ/2\ZZ$ embeds into $\PGL_2(k)$ over any field $k$ and the possible embeddings are of the form
$$
\rho_b\colon -1 \mapsto \left(\begin{array}{cc}
0 & b\\
1 & 0
\end{array}\right),
$$
up to conjugacy (see, e.g., \cite{Beau10}). We denote $\fourIdx{}{\rho_b}{}{}{\PP^1}$ simply by $\fourIdx{}{b}{}{}{\PP^1}$. Note that $\fourIdx{}{b}{}{}{\PP^1}$ and $\fourIdx{}{b'}{}{}{\PP^1}$ are isomorphic as $\ZZ/2\ZZ$-varieties if and only if $\ol{b} = \ol{b'}$. By \cite[Example 6]{Le07}, it follows that $\fourIdx{}{b}{}{}{\PP^1}$ is versal if and only if $b \in k^{\times 2}$.

\begin{cor} \label{Corollary_Cohom_2group}
Let $\rho_b$ be defined as above, let $K/k$ be a field extension and identify $H^1(K,\ZZ/2\ZZ)$ with $K^\times/K^{\times 2}$. Then $\Delta_{\rho_b}(\ol{c}) = [(c,b)_2]$ for all $c \in K^\times$.
\end{cor}

\begin{proof}
We need to show that $\rho_{b*}\colon K^\times/K^{\times 2} \to H^1(K,\PGL_2)$ maps $\ol{c}$ to $(c,b)_2$ for all $c \in K^{\times}$. Note that we may write $\rho_b = \rho_{(1,b)} \circ \phi$, where $\phi\colon \ZZ/2\ZZ \to (\ZZ/2\ZZ)^2$ sends $-1 \mapsto e_1$. Therefore we must have
$$
\rho_{b*}(\ol{c}) = \rho_{(1,b)*} \circ \phi_* (\ol{c}) = \rho_{(1,b)*}(\ol{c},\ol{1}) = (c,b)_2,
$$
where the last equality follows from Lemma \ref{Lemma_Cohom_Kleingroup}.
\end{proof}

\section{Cyclic and dihedral groups: Compressibility of $\PP^1$} \label{Sect_Incompressibility_Cyclic_Dihedral}

We set some notation for the remainder of the paper. Given an integer $n \geq 2$, let $\omega_n$ be a choice of a primitive $n$-th root of $1$, $\alpha_n = (\omega_n + \omega_n^{-1})/2$ and $\beta_n = \alpha_n^2-1$.

Recall that the groups $\ZZ/n\ZZ$ and $D_{2n}$ act faithfully on some curve of genus $0$ if and only if they act faithfully on $\PP^1$, which happens if and only if $\alpha_n \in k$ (see \cite{Beau10,Ga13} for details). If the latter condition does not hold, the existence of strongly incompressible curves for $\ZZ/n\ZZ$ and $D_{2n}$ follows from Theorem \ref{Thm_Incompressible_No_Action_Genus0}.

\begin{lemma} \label{Lemma_Compr_Cyclic_Dihedral}
Let $n \geq 3$ be any integer, let $k$ be a field containing $\alpha_n$, and define the embedding 
$\rho\colon D_{2n} \hookrightarrow \PGL_2(k)$ by sending
\begin{equation} \label{Eq_Embedding_Dihedral_1}
\sigma \mapsto \left(\begin{array}{cc}
\alpha_n + 1 & \beta_n\\
1 & \alpha_n + 1
\end{array}\right), \quad
\tau \mapsto \left(\begin{array}{cc}
1 & 0\\
0 & -1
\end{array}\right),
\end{equation}
where $\sigma, \tau$ are the usual generators of $D_{2n}$. Then $\fourIdx{}{\rho}{}{}{\PP^1}$ is not strongly incompressible.
\end{lemma}

\begin{proof}
We need to exhibit a $G$-equivariant map $\fourIdx{}{\rho}{}{}{\PP^1} \to \fourIdx{}{\rho}{}{}{\PP^1}$ that is not injective. Select a square root of $\beta_n$ (possibly in a quadratic extension of $k$) and define 
$$
Q = \left(\begin{array}{cc}
1 & 1\\
-\beta_n^{-1/2} & \beta_n^{-1/2}
\end{array} \right),
$$
in such a way that
$$
Q^{-1}\rho(\sigma)Q = \left(\begin{array}{cc}
1+\omega_n & 0\\
0 & 1+\omega_n^{-1}  
\end{array} \right), \quad Q^{-1}\rho(\tau)Q = \left(\begin{array}{cc}
0 & 1\\
1 & 0  
\end{array} \right).
$$
Let $F\colon \PP^1 \to \PP^1$ be given by $F(x:y) = (x^{n+1}:y^{n+1})$. A calculation shows that
$$
F\circ (Q^{-1}\rho(\sigma)Q) = (Q^{-1}\rho(\sigma)Q) \circ F
$$
and
$$
F\circ (Q^{-1}\rho(\tau)Q) = (Q^{-1}\rho(\tau)Q) \circ F.
$$

It follows that $Q \circ F \circ Q^{-1}$ is a $G$-equivariant map $\fourIdx{}{\rho}{}{}{\PP^1} \to \fourIdx{}{\rho}{}{}{\PP^1}$ defined over $k(\beta_n^{1/2})$. Explicitly, note that $Q \circ F \circ Q^{-1}$ sends $(x:y)$ to
$$
\left((x+\beta_n^{1/2}y)^{n+1} + (x-\beta_n^{1/2}y)^{n+1}:\beta_n^{-1/2} ((x+\beta_n^{1/2}y)^{n+1} - (x-\beta_n^{1/2}y)^{n+1})\right).
$$
In particular, it follows that $Q \circ F \circ Q^{-1}$ is actually defined over $k$. Since it has degree $n+1$, it is not injective and we are done.
\end{proof}

\begin{remark} \label{Rem_Compr_Cyclic_Dihedral}
Restricting the embedding \eqref{Eq_Embedding_Dihedral_1} to $\ZZ/n\ZZ$, the above lemma proves {\em a fortiori} that the projective line is not strongly incompressible as a $\ZZ/n\ZZ$-variety.
\end{remark}

\begin{prop} \label{Prop_Cyclic_Versal_Incompressibility}
Let $n \geq 2$ be any integer and let $k$ be a field containing $\omega_n$. Then there are no strongly incompressible $\ZZ/n\ZZ$-varieties.
\end{prop}

\begin{proof}
This is proved in \cite[Example 5]{Re04}; we supply a short alternative proof. Recall that the embedding $\rho\colon \ZZ/n\ZZ \hookrightarrow \PGL_2(k)$ sending a generator of $\ZZ/n\ZZ$ to the diagonal matrix $\mathrm{diag}(\omega_n,1)$, is generic, i.e., $\fourIdx{}{\rho}{}{}{\PP^1}$ is versal. Any faithful $\ZZ/n\ZZ$-variety can thus be $\ZZ/n\ZZ$-compressed to $\fourIdx{}{\rho}{}{}{\PP^1}$. Moreover, $\fourIdx{}{\rho}{}{}{\PP^1}$ is not strongly incompressible, as shown by the nontrivial $\ZZ/n\ZZ$-compression $(x:y) \mapsto (x^{n+1}:y^{n+1})$.
\end{proof}

The techniques introduced above can be used to show that there are no strongly incompressible varieties for odd cyclic and odd dihedral groups if they act faithfully on the projective line.

\begin{prop} \label{Prop_Odd_Dihedral_Incompressibility}
Let $n \geq 3$ be an odd integer, let $k$ be a field containing $\alpha_n$, and let $G$ be either $\ZZ/n\ZZ$ or $D_{2n}$. Then there are no strongly incompressible $G$-varieties.
\end{prop}

\begin{proof}
We focus on the case $G = D_{2n}$; the cyclic case follows along the same lines. Note that the embedding $\rho$ introduced in \eqref{Eq_Embedding_Dihedral_1} is generic for odd $n$, i.e., the $G$-variety $\fourIdx{}{\rho}{}{}{\PP^1}$ is versal (see \cite[Thm. 8]{Le07}). It follows that any faithful $G$-variety can be $G$-compressed to $\fourIdx{}{\rho}{}{}{\PP^1}$. It thus suffices to prove that $\fourIdx{}{\rho}{}{}{\PP^1}$ itself is not strongly incompressible. This follows directly from Lemma \ref{Lemma_Compr_Cyclic_Dihedral}.
\end{proof}

\section{Strongly incompressible curves for even cyclic groups}

Let $G$ = $\ZZ/n\ZZ$, where $n \geq 4$ is even, and let $k$ be a field containing $\alpha_n$. Define the embedding $\rho\colon G \hookrightarrow \PGL_2(k)$ by sending
\begin{equation} \label{Eq_embedding_cyclic}
\sigma \mapsto \left(\begin{array}{cc}
\alpha_n + 1 & \beta_n\\
1 & \alpha_n + 1
\end{array}\right),
\end{equation}
where $\sigma$ is a generator of $\ZZ/n\ZZ$ (see \cite{Beau10, Ga13}). Recall that this embedding is unique up to conjugacy. By the results in Proposition \ref{Prop_Cyclic_Versal_Incompressibility}, it remains to analyze the case where $\omega_n \not\in k$. (In this situation, $\fourIdx{}{\rho}{}{}{\PP^1}$ is not versal.) Interestingly, we will prove that there exist strongly incompressible $G$-curves under this assumption.

\begin{prop} \label{Prop_Even_Cyclic_Not_Versal_Incompressibility}
Let $k$ be a field such that $\alpha_n \in k$ and $\omega_n \not\in k$. Then there exists a strongly incompressible $G$-curve.
\end{prop}

\begin{proof}
Our goal is to construct a hyperelliptic curve endowed with a faithful $G$-action that cannot be $G$-compressed to any curve of genus $\leq 1$. The result then follows from Lemma \ref{Lemma_NoCompression_Small_Genus}.

Let $m = n/2$, let $K = k(t)$ be the rational function field and consider the exact sequence $1 \to \ZZ/2\ZZ \to G \to \ZZ/m\ZZ \to 1$. Then we obtain an exact sequence in cohomology
$$
\xymatrix{
1 \ar[r] & H^1(K,\ZZ/2\ZZ) \ar[r] & H^1(K,G) \ar[r] & H^1(K,\ZZ/m\ZZ).
}
$$
Consider the class $[\fourIdx{}{\rho}{}{}{\PP^1}] \in H^1(K,G)$. Its image in $H^1(K,\ZZ/m\ZZ)$ is equal to the class of the $\ZZ/m\ZZ$-variety $Y = \fourIdx{}{\rho}{}{}{\PP^1}/(\ZZ/2\ZZ)$ (which is abstractly isomorphic to $\PP^1$). Recall that $\sigma^m$ maps to 
$$
\left(\begin{array}{cc}
0 & \beta_n\\
1 & 0
\end{array}\right)
$$
under the embedding $\rho$. It thus follows that the quotient map $\fourIdx{}{\rho}{}{}{\PP^1} \to Y$ is given by $(x:y) \mapsto (x^2+\beta_n y^2 : 2xy)$.

Note that $H^1(K,\ZZ/2\ZZ) = K^\times/K^{\times 2}$ acts on the fiber of $[Y]$. To describe this action, write $L = k(\fourIdx{}{\rho}{}{}{\PP^1}) = k(u)$, which is a $G$-Galois extension of $K$. Then the subextension $L_0 = k(Y) = k(x)$ of $L$, where $x = \frac{u^2 + \beta_n}{2u}$, corresponds to the $\ZZ/m\ZZ$-torsor $Y \to Y/(\ZZ/m\ZZ)$. A short computation shows that $L$ can be obtained from $L_0$ by adjoining $\sqrt{x^2-\beta_n}$. An element $c \in K^\times/K^{\times 2}$ acts on the torsor $L/K$ by replacing $L = L_0(\sqrt{x^2-\beta_n})$ with $L_0(\sqrt{(x^2-\beta_n)f(t)})$, where $f(t) \in K$ is any representative of the class $c$. Recall that $t$ is a rational function of $x$, say $t = p(x)/q(x)$. Thus, in general we obtain the function field of a hyperelliptic curve $X$, which is naturally endowed with a faithful $G$-action. Our goal is to make a clever selection of the class $c$. Take $a \in k$, not a zero of $q$ or $q p' - p q'$, and let $c$ be the class of $f(t) = q(a)t - p(a)$. An explicit hyperelliptic equation for $X$ is given by $y^2 = s(x)$, where $s(x)$ is the squarefree part of 
$$
(x^2 - \beta_n) (q(a) p(x) - p(a) q(x)) q(x).
$$ 
By assumption, $s(a) = 0$, whence $(x,y) = (a,0)$ is a $k$-rational point of $X$.

We claim that $X$ cannot be $G$-compressed to any curve of genus $0$. First of all, such a curve would be forced to be $\fourIdx{}{\rho}{}{}{\PP^1}$ since $X$ has $k$-rational points. For the sake of contradiction, suppose that there exists a $G$-compression $X \to \fourIdx{}{\rho}{}{}{\PP^1}$. Regard this map as a $\ZZ/2\ZZ$-compression. As we saw before, $\fourIdx{}{\rho}{}{}{\PP^1}$ is isomorphic to $\fourIdx{}{\beta_n}{}{}{\PP^1}$ as a $\ZZ/2\ZZ$-variety. On the other hand, if we regard $X$ as a $\ZZ/2\ZZ$-variety, its class $[X] \in H^1(L_0,\ZZ/2\ZZ) = L_0^\times/L_0^{\times 2}$ is given by $\ol{s(x)}$. If we denote the restriction of $\rho$ to $\ZZ/2\ZZ$ by $\rho '$, we conclude using Corollary \ref{Corollary_Cohom_2group} that $\Delta_{\rho '}(X) = [(s(x),\beta_n)_2]$. This class must be trivial over $L_0 = k(x)$ by Corollary \ref{Cor_Maps_P1}. If we apply Lemma \ref{Lemma_Square_Quaternions} to the root $a$ of $s$, we obtain that $\beta_n \in k^{\times 2}$, i.e., $\omega_n - \omega_n^{-1} \in k$. However, this contradicts the fact that $\omega_n \not\in k$.

Next, we prove that $X$ cannot be $G$-compressed to any $G$-curve of genus $1$. Since $X$ has $k$-rational points, it suffices to prove that there is no $G$-compression $X \to E$, where $E$ is an elliptic curve endowed with a faithful $G$-action. Suppose there is such a $G$-compression and regard $G$ as a subgroup of $\Aut(E)$. By Lemma \ref{Lemma_Facts_Elliptic}(a), we can write $G \cong G_0 \times \pi(G)$, where $G_0 = G \cap E$ and $\pi(G) \subset \Aut_0(E)$. Since $G$ is cyclic, we conclude that $G_0$ and $\pi(G)$ are cyclic groups of relatively prime order. We claim that $\sigma^m$ (the unique element of order $2$ inside $G$) belongs to $G_0$, or equivalently that $G_0$ has even order. Suppose on the contrary that the order of $\pi(G)$ is even, i.e., $\pi(G) \cong \ZZ/d\ZZ$ for $d = 2$, $4$, or $6$. By Lemma \ref{Lemma_Facts_Elliptic}(d), the translation by $P_0 \in E$ and the automorphism $\alpha \in \Aut_0(E)$ commute if and only if $\alpha(P_0) = P_0$. Since $\pi(G)$ has even order, it contains the inversion map $P \mapsto -P$. Therefore any point of $E$ fixed by $\pi(G)$ has order dividing $2$. Since we are assuming that $G_0$ is a cyclic group of odd order that commutes with $\pi(G)$, it must be trivial. It follows that $G = \pi(G) \cong \Aut_0(E) \cong \ZZ/n\ZZ$ for $n = 4$ or $6$. By Lemma \ref{Lemma_Facts_Elliptic}(c), this contradicts the assumption that $k$ does not contain the appropriate roots of unity. We have proved that $\sigma^m \in G_0$, and hence acts as a translation on $E$. On the other hand, note that $\sigma^m$ fixes a $k$-rational point in $X$, namely $(x,y) = (a,0)$. Hence, it must also fix a point in $E$. This contradiction completes the proof.
\end{proof}

\begin{example}
We will explicitly construct the curve $X$ from the above proposition when $n=4$. In this case, $\alpha_4 = (\omega_4 + \omega_4^{-1})/2 = 0$ and $\beta_4 = -1$. It suffices to construct $X$ over the field of rational numbers. Note that the field extension $L_0/K$ satisfies $L_0 = k(x)$, $K = k(t)$ and $t = \frac{x^2-1}{2x}$. Taking $a=1$, the above construction yields the function field
$$
k(X) = L_0\left(\sqrt{(x^2+1) \frac{x^2-1}{x}}\right),
$$
whose corresponding hyperelliptic equation is $y^2 = x^5-x$. The $\ZZ/4\ZZ$-action on $X$ is given by $\sigma \cdot (x,y) \mapsto (-1/x,y/x^3)$, where $\sigma$ is a generator of $\ZZ/4\ZZ$. Note that this curve has genus $2$ and it was proved above that it does not map $G$-equivariantly to any curve of genus $\leq 1$. Hence it is an explicit example of a strongly incompressible $\ZZ/4\ZZ$-curve (recall that we are assuming $\omega_4 \not\in k$ throughout). In general, this procedure will not necessarily yield a strongly incompressible $G$-curve, but a $G$-curve that can be $G$-compressed to a strongly incompressible one. 
\end{example}

\section{Strongly incompressible curves for even dihedral groups}

\subsection{The Klein $4$-group}

Throughout this subsection, let $G$ denote the Klein $4$-group with generators $e_1, e_2$. Recall that $G$ acts faithfully on $\PP^1$ over any field $k$, but such an action is never versal. Our goal is to prove the following proposition.

\begin{prop} \label{Prop_Klein_Incompressibility}
The following are equivalent:

\begin{enumerate} [\upshape(i)]
\item There are no strongly incompressible $G$-curves over $k$.
\item $\mathrm{cd}_2(k) = 0$.
\end{enumerate}
\end{prop}

\begin{proof}[Proof of (ii) $\Rightarrow$ (i)] Suppose that $k$ has cohomological $2$-dimension $0$ and let $X$ be any faithful $G$-curve. The fixed field $K = k(X)^G$ is a transcendence degree $1$ extension of $k$, whence $\mathrm{cd}_2(K) \leq 1$ by \cite[Prop. II.4.2.11]{Se02}. Then, by \cite[Prop. II.2.3.4]{Se02}, it follows that $\Br_2(K)$ is trivial. Let $\rho\colon (\ZZ/2\ZZ)^2 \hookrightarrow \PGL_2(k)$ be any embedding. We claim that $X$ can be $G$-compressed to $\fourIdx{}{\rho}{}{}{\PP^1}$. Indeed, note that $\Delta_\rho(X)$ is the class of a quaternion algebra defined over $K$ and therefore trivial. The result then follows from Corollary \ref{Cor_Maps_P1}. To conclude the proof of the sufficiency, we need to prove that $\fourIdx{}{\rho}{}{}{\PP^1}$ is not strongly incompressible. We are free to select $\rho$ conveniently, so we may assume that $\rho$ is as in \eqref{Eq_embedding_Klein2} with $b = 1$. Then, it is obvious that $(x:y) \mapsto (x^3:y^3)$ is a $G$-compression of $\fourIdx{}{\rho}{}{}{\PP^1}$ to itself that is not birational.
\end{proof}

It remains to prove that (i) $\Rightarrow$ (ii). To achieve this, we need the following result.

\begin{prop} \label{Prop_Curve_Klein}
Let $P, Q \in k[x]$ be separable polynomials of degree $\geq 1$ satisfying the following conditions.
\begin{enumerate}[\upshape(i)]
\item $P$ and $Q$ have no common roots.
\item $P(0) \neq 0$, $Q(0) \neq 0$. 
\item There exists a root $x_1 \in k$ of $P$ (resp. $x_2 \in k$ of $Q$) such that $x_1Q(x_1) \in k^{\times 2}$ (resp. $x_2 P(x_2) \in k^{\times 2}$).
\end{enumerate}
Then the curve $X$ with function field $L = K(\sqrt{x P(x)},\sqrt{x Q(x)})$, where $K = k(x)$ is a rational function field, can be endowed with a faithful $G$-action such that every element of $G$ fixes at least one geometric point of $X$. 
\end{prop}

\begin{proof}
Let $\AA^3$ be the affine $3$-space over $k$ and let $Y \subset \AA^3$ be the affine variety cut out by the ideal $I = \langle y^2 - xP(x), z^2 - xQ(x)\rangle$. Note that $Y$ is an irreducible affine curve having a unique singularity at $(0,0,0)$ and its function field is precisely $L$. We can endow $Y$ with a faithful $G$-action by setting $e_1\cdot (x,y,z) = (x,-y,z)$ and $e_2\cdot (x,y,z) = (x,y,-z)$. This action can be lifted to the unique nonsingular projective curve $X$ which is birational to $Y$, in such a way that the natural birational isomorphism $X \dashrightarrow Y$ is $G$-equivariant. Note also that $X$ can be seen as a Galois $G$-cover of $\PP^1$ induced by the inclusion $K \hookrightarrow L$.

We claim that every element of $G$ fixes at least one geometric point of $X$. We first prove the assertion for $e_1 \in G$ to illustrate the procedure. Note that $A = (x_1,0,\sqrt{x_1Q(x_1)})$ is a nonsingular $k$-rational point of $Y$ fixed by $e_1$. Therefore, the natural $G$-equivariant rational map $Y \dashrightarrow X$ must be defined at the point $A$ and its image in $X$ is fixed by $e_1$ as desired. Analogously, we see that $B = (x_2,\sqrt{x_2P(x_2)},0)$ is a nonsingular $k$-rational point of $Y$ fixed by $e_2$ and the result follows along the same lines. 

It remains to prove that $e_1 e_2$ fixes a point in $X$. Unfortunately, the only fixed point of $e_1 e_2$ in $Y$ is $O = (0,0,0)$, which is not smooth. To overcome this difficulty, we consider the blowup of $\AA^3$ at the origin $O$ with exceptional divisor $E$ and consider the strict transform $Y'$ of $Y$. The $G$-action lifts naturally to $Y'$, in such a way that the birational morphism $Y' \to Y$ is $G$-equivariant. We claim that $Y'$ has a smooth point fixed by $e_1 e_2$, which has to be contained in $Y' \cap E$. Recall that 
$$
\mathrm{Bl}_O\AA^3 = \{((x, y, z), (t_0:t_1:t_2)) \in \AA^3 \times \PP^2\:|\: x t_1 = y t_0, x t_2 = z t_0, y t_2 = z t_1\}
$$
is covered by three affine charts $U_i = \{t_i \neq 0\}$ isomorphic to $\AA^3$. We pick coordinates $y, u = t_0/t_1, v = t_2/t_1$ in $U_1$ (so that $x = y u$ and $z = y v$) and compute $Y'$ in this coordinates. Any point in $Y' \cap U_1$ must satisfy the equations $y - uP(yu) = 0$ and $yv^2 - uQ(yu) = 0$. Moreover, note that the polynomial $Q(0)(y - uP(yu)) - P(0) (y v^2 - u Q(yu))$ is divisible by $y$ and consequently we obtain that
$$
Q(0) - P(0) v^2 - u^2 Q(0) P_1(y u) + u^2 P(0) Q_1(y u) = 0,
$$
for all points $(y,u,v) \in Y' \cap U_1$, where $P_1(x) = (P(x) - P(0))/x$ and $Q_1(x) = (Q(x) - Q(0))/x$. Then it is easy to see that the above three equations define $Y' \cap U_1$ and that $Y' \cap U_1 \cap E = \{(0, 0, \pm\sqrt{Q(0)/P(0)})\}$. (Actually one can see by looking at the other two charts that $Y' \cap E$ consists only of these two points.) We now look at the $G$-action on $Y' \cap U_1$. Note that $e_1 e_2 \cdot (y,u,v) = (-y,-u,v)$, since $e_1 e_2 \cdot (x,y,z) = (x,-y,-z)$ in $Y$. Therefore, the points $(0, 0, \pm\sqrt{P(0)/Q(0)})$ are fixed by $e_1 e_2$. Moreover, by applying the Jacobian criterion to the three polynomials defining $Y' \cap U_1$, one can show that both points are smooth; the details are left to the reader. Since the $G$-equivariant rational map $Y' \to Y \dashrightarrow X$ is defined at all smooth points, it follows that $e_1 e_2$ has a fixed point in $X$. 
\end{proof}

\begin{lemma} \label{Lemma_Curve_Constr_Genus1}
Let $X$ be any (smooth projective) $G$-curve obtained from Proposition \ref{Prop_Curve_Klein}. Then $X$ cannot be $G$-compressed to any curve of genus $1$.
\end{lemma}

\begin{proof}
Suppose that such a $G$-compression $X \to E$ exists. We may assume that $E$ is an elliptic curve since $X$ has $k$-rational points. By parts (a) and (b) of Lemma \ref{Lemma_Facts_Elliptic}, some element of $G$ must act freely on $E$. This contradicts the fact that every element of $G$ fixes a point in $X$.
\end{proof}

We are ready to prove that (i) $\Rightarrow$ (ii) in Proposition \ref{Prop_Klein_Incompressibility}. Suppose that $k$ does not have cohomological $2$-dimension $0$. We will produce a faithful $G$-curve that cannot be $G$-compressed to any $G$-curve of genus $\leq 1$ by using Proposition \ref{Prop_Curve_Klein}. The following well known lemma provides more manageable conditions on $k$.

\begin{lemma} \label{Lemma_Fields_cd2}
Let $k$ be a field. The following are equivalent:
\begin{enumerate} [\upshape(i)]
\item $\mathrm{cd}_2(k) = 0$.
\item $k$ is hereditarily quadratically closed, i.e., every algebraic extension of $k$ is quadratically closed.
\item $\xi$ is a square in $k(\xi)$ for every $\xi \in \ol{k}$.
\item $\Br_2(k(x)) = 0$.
\end{enumerate}
\end{lemma}

\begin{proof}
The equivalence (ii) $\Leftrightarrow$ (iii) is straightforward and left to the reader, while (i) $\Leftrightarrow$ (ii) follows directly from \cite[Lemma 2]{EW87}. We now prove that (i) $\Rightarrow$ (iv). If $\mathrm{cd}_2(k) = 0$, it follows from \cite[Prop. II.4.1.11]{Se02} that $\mathrm{cd}_2(k(x)) \leq 1$. Then we conclude that $\Br_2(k(x)) = 0$ by \cite[Prop. II.2.3.4]{Se02}. To complete the proof, it suffices to show that (iv) $\Rightarrow$ (iii). Suppose that (iv) holds, but there exists $\xi \in \ol{k}$, which is not a square in $k(\xi)$. Let $h \in k[x]$ be the minimal polynomial of $\xi$ over $k$. The quaternion algebra $(x,h(x))_2$ must be split over $k(x)$, which implies that $\xi$ is a square over $k(\xi)$ by Lemma \ref{Lemma_Square_Quaternions}. This contradiction completes the proof.
\end{proof}

\begin{constr} \label{Constr_Curve_Klein}
In view of Lemma \ref{Lemma_Fields_cd2}, given that $\mathrm{cd}_2(k) \neq 0$, we can choose an element $\xi$ algebraic over $k$, which is not a square in $k(\xi)$. Let $h \in k[x]$ be the minimal polynomial of $\xi$ over $k$ and define polynomials
\begin{eqnarray*}
P(x) &=& (x-\alpha) \left(\frac{(x-\alpha-1)h(x-\alpha) + (\alpha+1)h(-\alpha)}{(\alpha+1) h(-\alpha) x}\right),\\
Q(x) &=& \alpha(\alpha+1-x)h(0)h(x-\alpha),
\end{eqnarray*}
where $\alpha \in k$ is taken such that $P$ has no multiple roots, $P(0) \neq 0$ and $Q(0) \neq 0$. (It is not hard to see that such a selection of $\alpha$ is always possible.) We conclude that $P$ and $Q$ satisfy the conditions of Proposition \ref{Prop_Curve_Klein}; in what follows, let $X$ denote the corresponding curve.
\end{constr}

\begin{lemma} \label{Lemma_Curve_Constr_Genus0}
Let $X$ be the curve obtained in Construction \ref{Constr_Curve_Klein}. Then there is no $G$-compression $X \to Y$, where $Y$ is a curve of genus $0$.
\end{lemma}

\begin{proof}
By construction, $X$ has $k$-rational points. Hence, such a $G$-compression could only be possible if $Y \cong \fourIdx{}{\rho}{}{}{\PP^1}$ for some embedding $\rho\colon G \hookrightarrow \PGL_2(k)$.
From Proposition \ref{Prop_Curve_Klein}, we observe that $k(X)^G = K = k(x)$. Moreover, note that the class $[X] \in H^1(K,G)$ corresponds to $(\ol{xP(x)},\ol{xQ(x)}) \in (K^\times/K^{\times 2})^2$. By Lemma \ref{Lemma_Cohom_Kleingroup}, we obtain that $\Delta_\rho(X) = [(a x P(x),b x Q(x))_2] \in \Br(K)$ for some $(\ol{a},\ol{b}) \in (k^\times/k^{\times 2})^2$ such that $(a,b)_2$ is split.

Suppose that there exists a $G$-compression $X \to \fourIdx{}{\rho}{}{}{\PP^1}$. Then, by Corollary \ref{Cor_Maps_P1}, the quaternion algebra $(a x P(x), b x Q(x))_2$ must be split over $K$. Applying Lemma \ref{Lemma_Square_Quaternions} to the roots $\alpha+1$ and $\alpha + \xi$ of $b x Q(x)$, we obtain that $a \in k^{\times 2}$ and $a \xi \in k(\xi)^{\times 2}$, respectively. This contradicts the assumption that $\xi$ is not a square in $k(\xi)$.
\end{proof}

To finish the proof of Proposition \ref{Prop_Klein_Incompressibility}, we use Lemma \ref{Lemma_Curve_Constr_Genus1} and Lemma \ref{Lemma_Curve_Constr_Genus0} to conclude that $X$ cannot be $G$-compressed to any curve of genus $\leq 1$. Thus, it follows from Lemma \ref{Lemma_NoCompression_Small_Genus} that there exist strongly incompressible $G$-curves if $\mathrm{cd}_2(k) > 0$. The proof is now complete.

\subsection{Even dihedral groups of order $\geq 8$}

In this subsection, $G$ will always denote the dihedral group $D_{2n}$, where $n \geq 4$ is an even integer. A result similar to Proposition \ref{Prop_Klein_Incompressibility} holds in this case.

\begin{prop} \label{Prop_Even_Dihedral_Incompressibility} Let $k$ be a field such that $\alpha_n \in k$. Then there exist no strongly incompressible $G$-curves defined over $k$ if and only if $\mathrm{cd}_2(k)=0$.
\end{prop}

\begin{proof}
Suppose first that $\mathrm{cd}_2(k)=0$. Similarly to the proof of Proposition \ref{Prop_Klein_Incompressibility}, it follows that any faithful $G$-curve $X$ can be $G$-compressed to $\fourIdx{}{\rho}{}{}{\PP^1}$, where the embedding $\rho\colon G \hookrightarrow \PGL_2(k)$ is as in \eqref{Eq_Embedding_Dihedral_1}. Moreover, it follows from Lemma \ref{Lemma_Compr_Cyclic_Dihedral} that $\fourIdx{}{\rho}{}{}{\PP^1}$ is not strongly incompressible.

To prove the converse, assume that $\mathrm{cd}_2(k)>0$. We must show that there exists a strongly incompressible $G$-curve under this assumption. We first study the special case where $\omega_n \not\in k$, i.e., $\beta_n$ is not a square in $k$. We only sketch the argument, as it is very similar to the proof of Proposition \ref{Prop_Even_Cyclic_Not_Versal_Incompressibility}. Using the cohomology sequence associated to the central exact sequence $1 \to \ZZ/2\ZZ \to G \to D_n \to 1$, one can construct a hyperelliptic $G$-curve $X$, having a $k$-rational point fixed by the hyperelliptic involution. Suppose that $X$ can be $G$-compressed to $\fourIdx{}{\eta}{}{}{\PP^1}$, where $\eta$ is any embedding $G \hookrightarrow \PGL_2(k)$. Regard the $G$-compression as a $\ZZ/2\ZZ$-compression \emph{with respect to the center of $G$}. As a $\ZZ/2\ZZ$-variety, $\fourIdx{}{\eta}{}{}{\PP^1}$ is isomorphic to $\fourIdx{}{\beta_n}{}{}{\PP^1}$. As in the proof of Proposition \ref{Prop_Even_Cyclic_Not_Versal_Incompressibility}, we must have $\beta_n \in k^{\times 2}$, contradicting our assumption. Similarly, it follows that $X$ cannot be $G$-compressed to any curve of genus $1$. By Lemma \ref{Lemma_NoCompression_Small_Genus}, there exists a strongly incompressible $G$-curve in this case.

In what follows, assume that $\omega_n \in k$. By Lemma \ref{Lemma_Fields_cd2}, there exists $\xi \in \ol{k}$ such that $\xi$ is not a square in $k(\xi)$. Using this information, we construct a hyperelliptic $G$-curve $X$ that cannot be $G$-compressed to any curve of genus $\leq 1$. Let $m=n/2$, let $K=k(t)$ be the rational function field and consider the exact sequence $1 \to \ZZ/2\ZZ \to G \to D_{2m} \to 1$, which yields an exact sequence in cohomology
$$
\xymatrix{
1 \ar[r] & H^1(K,\ZZ/2\ZZ) \ar[r] & H^1(K,G) \ar[r] & H^1(K,D_{2m}).
}
$$
Also, let $\rho\colon G \hookrightarrow \PGL_2(k)$ be given by
\begin{equation*}
\sigma \mapsto \left(\begin{array}{cc}
\omega_n & 0\\
0 & 1
\end{array}\right), \quad
\tau \mapsto \left(\begin{array}{cc}
0 & 1\\
1 & 0
\end{array}\right).
\end{equation*}
(Note that this embedding is conjugate to the one in \eqref{Eq_Embedding_Dihedral_1}.) The image of the class $[\fourIdx{}{\rho}{}{}{\PP^1}] \in H^1(K,G)$ in $H^1(K,D_{2m})$ is the class of the $D_{2m}$-variety $Y = \fourIdx{}{\rho}{}{}{\PP^1}/(\ZZ/2\ZZ)$ and the quotient map $\fourIdx{}{\rho}{}{}{\PP^1} \to Y$ is simply $(x:y) \mapsto (x^2:y^2)$.

Let $L = k(\fourIdx{}{\rho}{}{}{\PP^1}) = k(u)$, which is a $G$-Galois extension of $K$. The $D_{2m}$-torsor $Y \to Y/D_{2m}$ corresponds to a $D_{2m}$-Galois extension $L_0/K$, where $L_0 = L^{\ZZ/2\ZZ}$. Hence, we can write $L_0 = k(x)$, where $x = u^2$. An element $c \in K^\times/K^{\times 2}$ acts on the fibre of $[Y]$ by replacing $L = L_0(\sqrt{x})$ by $L_0(\sqrt{x f(t)})$, where $f(t)$ is any representative of $c$. 

A simple computation shows that we can take $t = (x^m + x^{-m})/2$, which is a polynomial in $(x+x^{-1})/2$. Denote such polynomial by $T_m$. Select $\gamma, \delta \in \ol{k}$ such that $\gamma^2 = 1+\xi$ and $\delta^2=1+\xi^{-1}$. Then, let $f \in k[t]$ be a separable polynomial vanishing at the points $T_m(\gamma)$ and $T_m(\delta)$. We define $c$ to be the class of $f(t)$ in $K^\times/K^{\times 2}$ and denote by $X$ the corresponding hyperelliptic curve, whose function field is equal to $L_0(\sqrt{x f(t)})$. Multiplying $f$ by an element of $k$ if necessary, we may assume that $X$ has a $k$-rational point fixed by the hyperelliptic involution.

The conjugacy classes of embeddings $D_{2n} \hookrightarrow \PGL_2(k)$ are parametrized by the set $D(\langle 1,-\beta_n\rangle)$ of nonzero square classes represented by the binary quadratic form $x^2-\beta_n y^2$. The correspondence is as follows: to the class $\ol{a}$ of the element $a = x^2-\beta_n y^2$ ($x, y \in k$), we assign
\begin{equation} \label{Eq_embedding_dihedral_general}
\rho_a\colon \sigma \mapsto \left(\begin{array}{cc}
\alpha_n + 1 & \beta_n\\
1 & \alpha_n + 1
\end{array}\right), \quad
\tau \mapsto \left(\begin{array}{cc}
x & -y\beta_n\\
y & -x
\end{array}\right),
\end{equation}
where $\sigma, \tau \in D_{2n}$ satisfy $\sigma^n = \tau^2 = (\sigma\tau)^2 = 1$ (see \cite{Beau10, Ga13}). We claim that $X$ cannot be $G$-equivariantly compressed to  $\fourIdx{}{\eta}{}{}{\PP^1}$, where $\eta$ is the embedding \eqref{Eq_embedding_dihedral_general} for some class $\ol{a}$. Note that we are assuming that $\beta_n$ is a square in $k$, so the binary form $\langle 1, -\beta_n \rangle$ is universal and therefore, $\ol{a}$ can be any element of $k^\times/k^{\times 2}$. Since we have $\omega_n \in k$, the embedding $\eta$ can be conjugated to
$$
\sigma \mapsto \left(\begin{array}{cc}
\omega_n & 0\\
0 & 1
\end{array}\right), \quad
\tau \mapsto \left(\begin{array}{cc}
0 & a\\
1 & 0
\end{array}\right),
$$
so we may assume that $\eta$ is of this form. Suppose for the sake of contradiction that there exists a $G$-compression $X \to \fourIdx{}{\eta}{}{}{\PP^1}$ and regard it as a $(\ZZ/2\ZZ)^2$-compression. It follows that $\Delta_{\eta}(X) = 1$. We compute this Brauer class explicitly. For simplicity, write the function field of $X$ in the form 
$$
L' = k(x,y)/(y^2 - x f((x^m + x^{-m})/2)).
$$ 
It is not hard to see that 
\begin{eqnarray*}
L'^{\langle \sigma^m \rangle} &=& k(x),\\ 
L'^{\langle \tau \rangle} &=& k((x + x^{-1})/2, y(1+x^{-1}))/(y^2 - x f((x^m + x^{-m})/2)),\\ 
L'^{\langle \sigma^m, \tau \rangle} &=& k((x + x^{-1})/2).
\end{eqnarray*}
Let $s = (x + x^{-1})/2$. Note that
$$
y^2(1+x^{-1})^2 = x f((x^m + x^{-m})/2) (1 + 2 x^{-1} + x^{-2}) = (2s+2) f(T_m(s)),
$$ 
whence $L'^{\langle \tau \rangle}$ is obtained from $k(s)$ by adjoining $\sqrt{(2s+2) f(T_m(s))}$. On the other hand, note that $L'^{\langle \sigma^m \rangle}$ is obtained from $k(s)$ by adjoining $\sqrt{s^2-1}$. Regarding $X$ as a $(\ZZ/2\ZZ)^2$-variety, its class $[X] \in H^1(k(s),(\ZZ/2\ZZ)^2)$ is thus equal to $(\ol{2(s+1) f(T_m(s))},\ol{s^2-1})$. By Lemma \ref{Lemma_Cohom_Kleingroup}, we easily conclude that $\Delta_\eta(X) = [(2a(s+1) f(T_m(s)), s^2-1)_2]$. (In the notation of Lemma \ref{Lemma_Cohom_Kleingroup}, note that $\fourIdx{}{\eta}{}{}{\PP^1} \cong \fourIdx{}{(a,1)}{}{}{\PP^1}$ as a $(\ZZ/2\ZZ)^2$-variety.) We may assume that $(s+1) f(T_m(s))$ is separable, by replacing $\xi$ by another element in $\xi\cdot k^{\times 2}$ if necessary. Applying Lemma \ref{Lemma_Square_Quaternions} to $\gamma$, we obtain that $\xi$ is a square in $k(\gamma)$. It follows that $k(\gamma) = k(\sqrt{\xi})$, since $[k(\sqrt{\xi}):k(\xi)] = 2$ by assumption. Hence, we can write $\gamma = l_1 + l_2 \sqrt{\xi}$ for some $l_1, l_2 \in k(\xi)$. Squaring, we obtain that $\xi + 1 = l_1^2 + l_2^2 \xi + 2 l_1 l_2 \sqrt{\xi}$, whence $l_1 l_2 = 0$. If $l_2 = 0$, it follows that $\xi + 1$ is a square in $k(\xi)$, contradicting the fact that $[k(\gamma):k(\xi)] = 2$. Therefore, we must have $l_1 = 0$, which implies that $1 + \xi^{-1}$ is a square in $k(\xi)$, i.e., $\delta \in k(\xi)$. However, applying Lemma \ref{Lemma_Square_Quaternions} to $\delta$ implies that $\xi^{-1}$ (and hence $\xi$) is a square in $k(\delta) = k(\xi)$, which contradicts our initial assumption. This proves that a $G$-compression $X \to \fourIdx{}{\eta}{}{}{\PP^1}$ is not possible. Since $X$ has $k$-rational points, we conclude that $X$ cannot be $G$-compressed to any curve of genus $0$.

It remains to prove that $X$ cannot be $G$-compressed to any curve of genus $1$. Suppose there is such a $G$-compression $X \to E$. By construction, the hyperelliptic involution of $X$, namely $\sigma^m$, fixes some $k$-rational point of $X$. Hence, $\sigma^m$ must fix some $k$-rational point of $E$, which we may assume to be an elliptic curve. We adopt the notation of Lemma \ref{Lemma_Facts_Elliptic}(a), where $\pi\colon \Aut(E) \to \Aut_0(E)$ denotes the natural projection. Since $\Aut_0(E)$ is abelian, the relation $(\sigma \tau)^2 = 1$ implies that $\pi(\sigma \tau)^2 = \pi(\sigma^2) \pi(\tau^2) = \pi(\sigma^2) = 1$. It follows that $\sigma^2$ acts as a translation on $E$. We claim that $\sigma$ acts as a translation as well. By Lemma \ref{Lemma_Facts_Elliptic}(a), we may write $\sigma = \tau_{P_0} \circ \alpha$, where $\tau_{P_0}$ denotes the translation by $P_0 \in E$ and $\alpha \in \Aut_0(E)$. Since $\sigma^2$ is a translation, it follows that $\sigma^2(P) - P = \alpha^2(P) - P + \alpha(P_0) + P_0$ must be constant for all $P \in E$. This implies that the isogeny $\alpha^2 - \mathrm{id}$ is constant, so it is the zero map. This proves that $\alpha$ has order $2$ in $\Aut_0(E)$, whence $\alpha$ is the inversion map $P \mapsto -P$. This implies that $\sigma^2 = \mathrm{id}$ in $\Aut(E)$, which is a contradiction because $\sigma$ has order $n \geq 4$. We have proved that $\sigma$ acts as a translation on $E$, and therefore so does $\sigma^m$. Hence $\sigma^m$ cannot fix any point of $E$.
\end{proof}

\section{Polyhedral groups} \label{Sect_Polyhedral}

It remains to study the incompressibility of curves endowed with a faithful action of a polyhedral group $G$, i.e., $G = A_4$, $S_4$, or $A_5$.

\subsection{Serre's cohomological invariant} 

Let $\widehat{G}$ be the binary polyhedral group associated to $G$. If $G$ is an alternating group, then $\widehat{G}$ coincides with the unique nontrivial central extension of $G$ by $\ZZ/2\ZZ$. If $G = S_4$, then $\widehat{G}$ is the unique central extension of $G$ by $\ZZ/2\ZZ$ in which transpositions and products of disjoints transpositions lift to elements of order $4$. (This is not the double cover studied in \cite{Se84}, in which transpositions lift to involutions). We have a central exact sequence
$$
\xymatrix{
1 \ar[r] & \ZZ/2\ZZ \ar[r] & \widehat{G} \ar[r] & G \ar[r]  & 1
},
$$
which yields a corresponding sequence in cohomology
$$
\xymatrix{H^1(K,\widehat{G}) \ar[r]  & H^1(K,G) \ar[r]^-{\widehat{\Delta}} & \Br_2(K),
}
$$
for any field extension $K/k$. Note that $\widehat{\Delta} \colon H^1(K,G) \to H^2(K,\ZZ/2\ZZ) = \Br_2(K)$ defines a cohomological invariant. If $X$ is a faithful primitive $G$-variety and $L = k(X)^G$, we denote the Brauer class associated to $[X] \in H^1(L,G)$ by $\widehat{\Delta}(X)$. Note that $\widehat{\Delta}(X)$ is trivial if and only if $[X]$ lifts to a $\widehat{G}$-torsor $[\widehat{X}] \in H^1(L,\widehat{G})$. The following result follows from the definition of cohomological invariant.

\begin{prop} \label{Prop_Invariant_Polyhedral}
Let $X, Y$ be faithful primitive $G$-varieties and suppose that there exists a $G$-compression $f\colon X \dashrightarrow Y$. Let $i\colon k(Y)^G \hookrightarrow k(X)^G$ be the natural inclusion induced by $f$ and define $i_*\colon \Br_2(k(Y)^G) \to \Br_2(k(X)^G)$ as the corresponding functorial map. Then $i_*(\widehat{\Delta}(Y)) = \widehat{\Delta}(X)$.
\end{prop}

\begin{proof}
Left to the reader.
\end{proof}

J.-P. Serre has described an effective way to compute $\widehat{\Delta}$. An element of $H^1(K,G)$ can be viewed as (the isomorphism class of) an {\' e}tale $K$-algebra $E$, which has trivial discriminant if $G$ is alternating. Then we have the following result.

\begin{prop} [{cf. \cite[Th. 1]{Se84}}] \label{Prop_Computation_Invar_Polyhedral}
Let $q_{E}$ is the trace form of $E/K$. Then
$$
\widehat{\Delta}(E) = w_2(q_E) + [(-2, d_E)_2],
$$
where $w_2$ denotes the second Stiefel-Whitney class and $d_E$ is the discriminant of $E$.
\end{prop}

\begin{proof}
See \cite[Th. 1]{Se84} or \cite[\S 2]{Vi88}.
\end{proof}

\begin{remark} \label{Remark_Equivalence_Invariants}
If the field $k$ satisfies some additional conditions, we may view $\widehat{\Delta}$ as a particular case of the cohomological invariant defined in Section \ref{Sect_Equiv_ProjSpaces}. Suppose that the following assumptions hold.

\begin{enumerate} [\upshape(i)]
\item There exists an embedding $\rho\colon G \hookrightarrow \PGL_2$. This is the case if and only if $-1$ is the sum of two squares over $k$, with the additional requirement that $\sqrt{5} \in k$ if $G = A_5$ (see \cite{Beau10}).
\item There exists an embedding $\ol{\rho}\colon \widehat{G} \hookrightarrow \GL_2$ that fits in a commutative diagram
$$
\xymatrix{1 \ar[r] & \GG_m \ar[r] & \GL_2 \ar[r]  & \PGL_2 \ar[r] & 1\\
1 \ar[r] & \ZZ/2\ZZ \ar[r] \ar@{^{ (}->}[u]  & \widehat{G} \ar@{^{ (}->}[u]_-{\ol{\rho}} \ar[r] & G \ar@{^{ (}->}[u]_-{\rho} \ar[r] & 1}
$$
This is automatic if $G$ is alternating. In the case $G = S_4$, it is true if and only if $\sqrt{2} \in k$. 
\end{enumerate}

Passing to cohomology in the above diagram, we conclude that $\widehat{\Delta}$ coincides with $\Delta_{\rho}$, if we regard both their images to lie in the Brauer group. 
\end{remark}

\subsection{Computation of the invariant for curves of genus $\leq$ 1}

We first compute the cohomological invariant $\widehat{\Delta}$ for polyhedral actions on curves of genus $0$. Recall that, up to equivariant birational isomorphism, there is only one action of a polyhedral group $G$ on a curve of genus $0$ (see \cite{Beau10, Ga13}). In what follows, let $q_0(x,y,z) = x^2+y^2+z^2$ and denote by $X_0 \subset \PP^2$ the corresponding quadric. Then, $G$ acts on $X_0$ via the standard embedding $\rho\colon G \hookrightarrow \SO(q_0)$ as a rotation group. If $G = A_4$ or $S_4$, the action is defined over any field $k$, while for $G = A_5$ the action is defined over $k$ if and only if $\sqrt{5} \in k$. Recall also that $K:=k(X_0)^G$ is isomorphic to a rational function field, i.e., $X_0/G \cong \PP^1$.

\begin{prop} \label{Prop_Invariant_Polyhedral_0}
\begin{enumerate}[\upshape(a)]
\item If $G$ is alternating, then $\widehat{\Delta}(X_0) = [(-1,-1)_2]$ in $\Br_2(K)$.
\item If $G = S_4$, then $\widehat{\Delta}(X_0) = [(-1,-1)_2] + [(2,t)_2]$ in $\Br_2(K)$, where $t$ is some generator of $K/k$.
\end{enumerate}
\end{prop}

\begin{proof}
(a) Let $k'/k$ be a field extension, and suppose that $q_0$ is isotropic over $k'$. We claim that $\widehat{\Delta}(X_0')$ is trivial in $\Br_2(k'(X_0')^G)$, where $X_0' = X_0 \times_{\Spec(k)} \Spec(k')$. Indeed, note that $\PGL_2 \cong \SO(q_0)$ over $k'$, whence there exists an embedding $\rho\colon G \hookrightarrow \PGL_2$ defined over $k'$ and a $G$-equivariant isomorphism $X_0' \cong \fourIdx{}{\rho}{}{}{\PP^1}$. It follows from Remark \ref{Remark_Equivalence_Invariants} that $\widehat{\Delta}(X_0') = \Delta_\rho(X_0') = \Delta_\rho(\fourIdx{}{\rho}{}{}{\PP^1})$, which is trivial by Lemma \ref{Lemma_Class_PV}. This completes the proof of the claim.

Let $E$ be the \' etale algebra corresponding to the class $[X_0] \in H^1(K,G)$. Then $q_E \cong \langle 1, a, b, ab \rangle$ for some $a,b \in K$ if $G = A_4$ (resp. $q_E \cong \langle 1, a, b, c, abc \rangle$ for some $a,b,c \in K$ if $G = A_5$). It follows that $\widehat{\Delta}(X_0) = w_2(q_E) = [(-a,-b)_2]+[(-1,-1)_2]$ if $G = A_4$ (resp. $[(-ac,-bc)_2]+[(-1,-1)_2]$ if $G = A_5$). In any case, we can write $\widehat{\Delta}(X_0) = [(u,v)_2]+[(-1,-1)_2]$ for some $u,v \in K$, so it suffices to prove that $(u,v)_2$ is split over $K$. Since $q_0$ is isotropic over $k':=k(s,t)/(s^2+t^2+1)$ and $(-1,-1)_2$ splits over $k'$, it follows from the previous paragraph that $(u,v)_2$ splits over $k'(X_0')^G \cong K(s,t)/(s^2+t^2+1)$. Equivalently, the Pfister form $\langle 1, -u, -v, uv \rangle$ is hyperbolic over $K(s,t)/(s^2+t^2+1)$, which is the function field of the quadratic form $\langle 1,1,1 \rangle$ defined over $K$. By \cite[Th. X.4.5]{Lam05}, either $\langle 1, -u, -v, uv \rangle$ is isotropic (hence hyperbolic) over $K$, or $\langle 1, -u, -v, uv \rangle \cong \langle 1, 1, 1, 1 \rangle$ over $K$. Equivalently, either $(u,v)_2$ splits, or $(u,v)_2 \cong (-1,-1)_2$. The former case yields the desired result, while the latter implies that $\widehat{\Delta}(X_0)$ is trivial. Hence, it suffices to prove that $\widehat{\Delta}(X_0)$ is nontrivial whenever $(-1,-1)_2$ is not split over $K$ (equivalently over $k$, since $K$ is purely transcendental over $k$). 

Assume for the sake of contradiction that $(-1,-1)_2$ is not split over $k$ and $\widehat{\Delta}(X_0)$ is trivial. This implies that $[X_0]$ comes from a class in $H^1(K,\widehat{G})$, i.e., there exists a faithful primitive $\widehat{G}$-variety $\widehat{X_0}$ such that $\widehat{X_0}/(\ZZ/2\ZZ)$ is birationally isomorphic to $X_0$ as a $G$-variety. Note that $\widehat{X_0}$ must be geometrically irreducible since $1 \to \ZZ/2\ZZ \to \widehat{G} \to G \to 1$ is not split. Thus, we may assume without loss of generality that $\widehat{X_0}$ is a (smooth projective) $\widehat{G}$-curve, endowed with a $2$-$1$ quotient morphism $\widehat{X_0} \to X_0$. It follows that $\widehat{X_0}$ is a hyperelliptic curve (in the sense that its canonical divisor is not very ample). Moreover, note that $\Aut(\widehat{X_0})(\ol{k})$ contains $\widehat{G}$, which equals $\SL_2(\mathbb{F}_3)$ if $G = A_4$ (resp. $\SL_2(\mathbb{F}_5)$ if $G = A_5$). By \cite[Table 1]{Sh03}, it follows that the genus of $\widehat{X_0}$ is even. However, it is well known that this implies that $\widehat{X_0}/(\ZZ/2\ZZ) = X_0$ has a $k$-rational point (see, e.g., \cite[\S 2.1]{Me91}), which is equivalent to the splitting of $(-1,-1)_2$ over the field $k$. This contradiction concludes the proof.

(b) Note that $S_4$ embeds into $\SO(q_0)$ as the matrices of the form $DP$, where $D$ is diagonal with entries $\pm 1$ and $P$ is a permutation matrix. (There are $24$ such matrices of determinant $1$.) The \'etale $K$-algebra corresponding to $[X_0] \in H^1(K,S_4)$ is the field extension $k(X_0)^{H}/K$, where $H$ is any copy of $S_3$ inside $S_4$. For convenience, we choose the subgroup $H$ generated by
$$
\sigma = \left(\begin{array}{ccc}
0 & 1 & 0\\
0 & 0 & 1\\
1 & 0 & 0
\end{array}\right), \quad
\tau = \left(\begin{array}{ccc}
0 & 0 & -1\\
0 & -1 & 0\\
-1 & 0 & 0
\end{array}\right).
$$
Note that $S_4 = V \rtimes H$, where $V$ is the subgroup of diagonal matrices inside $S_4$. 

We write $k(X_0) = k(a,b)/(a^2+b^2+1)$, where $a = x/z$ and $b = y/z$ in the usual coordinates of $X_0$. Note that $\sigma(a) = b/a$ and $\sigma(b) = 1/a$, while $\tau(a)=1/a$ and $\tau(b) = b/a$. An easy computation then shows that $k(X_0)^H = k(\alpha)$, where $\alpha = a + b/a + 1/b + 1/a + b + a/b$. By Galois Theory, the minimal polynomial of $\alpha$ over $K$ is equal to 
$$
P(Y) = \prod_{g \in V} (Y - g(\alpha)) = Y^4 - 6 Y^2 + 8 Y + t + 24,
$$
where
$$
t = \frac{(a-1)^2 (a+1)^2 (2 a^2+1)^2 (a^2+2)^2}{a^4 (a^2+1)^2}
$$
is a generator of $K/k$, which proves that $k(X_0)^H = K[Y]/(p(Y))$. By a simple computation, the trace form of $K[Y]/(p(Y))$ over $K$ is isomorphic to $\langle 1, 3, -(t+27), -3t(t+27) \rangle$. It follows that its Stiefel-Whitney class is equal to $[(-3t,t(t+27))_2] + [(-1,-t)_2]$. The first quaternion algebra is split over $K$ because $(-3t) 3^2 + t(t+27) = t^2$. It follows that $\widehat{\Delta}(X_0) = [(-1,-t)_2] + [(-2,t)_2] = [(-1,-1)_2] + [(2,t)_2]$. The proof is complete.
\end{proof}

We now focus our attention on polyhedral actions on curves of genus $1$. In this case, we only need to consider $A_4$-actions, since $S_4$ and $A_5$ cannot act faithfully on curves of genus $1$.

\begin{prop} \label{Prop_Genus_1_Alternating}
Let $C$ be a curve of genus $1$ endowed with a faithful $A_4$-action defined over a field $k$. Then the following properties hold.
\begin{enumerate} [\upshape(a)]
\item The Jacobian $E \cong \mathrm{Pic}^0(C)$ has $j$-invariant $0$.
\item The elliptic curve $E$ can be endowed with a faithful $A_4$-action defined over $k$. 
\item The curve $C$ is $A_4$-equivariantly isomorphic to $E$ over some extension $k'/k$ of odd degree.
\item We have the equality $\widehat{\Delta}(C) = [(-1,-1)_2]$ in 
$\Br_2(k(C)^{A_4})$.
\end{enumerate}
\end{prop}

\begin{proof}
We will extensively use the results and notation from \cite[\S X.3]{Sil09} (see also \cite{LT58}). Recall that $C$ is a principal homogeneous space under $E$. A $k$-automorphism $g\colon C \to C$ induces a group automorphism of $\mathrm{Pic}^0(C)$, hence also a $k$-automorphism $\hat{g}\colon E \to E$ fixing the origin. Explicitly, it is not hard to see that $\hat{g}(P) = g(p_0+P)-g(p_0)$, where the definition is independent of $p_0 \in C(\ol{k})$. Note also that $\hat{g}$ is the identity if and only if $g$ is a translation by an element of $E(k)$. This proves that we have an exact sequence
$$
\xymatrix{
1 \ar[r] & E(k) \ar[r] & \Aut(C)(k) \ar[r]^-\pi & \Aut_0(E)(k).
}
$$
Regard $A_4$ as a subgroup of $\Aut(C)(k)$. It follows that $E(k) \cap A_4 \cong (\ZZ/2\ZZ)^2$ and $\pi(A_4) = \ZZ/3\ZZ \subset \Aut_0(E)(k)$. By Lemma \ref{Lemma_Facts_Elliptic}(b), it follows that $j(E) = 0$.

We now proceed with the proof of part (b). Note that $E(k)$ contains a subgroup isomorphic to $(\ZZ/2\ZZ)^2$, whence the $2$-torsion points of $E$ are $k$-rational. Using Lemma \ref{Lemma_Facts_Elliptic}(b), we conclude from part (a) that $k$ contains a primitive third root of unity $\omega_3$ and $\Aut_0(E)(k) = \ZZ/6\ZZ$. We now explicitly construct the $A_4$-action on $E$. Since $E$ has $j$-invariant $0$, it has a Weierstrass equation $y^2 = x^3 + b$ for some $b \in k^\times$. Let the normal subgroup $V = (\ZZ/2\ZZ)^2 \subset A_4$ act on $E$ via translation by $2$-torsion points (as it does on $C$ as well). Then we can write $A_4 = V \rtimes H$, and let $H \cong \ZZ/3\ZZ$ act on $E$ by $\alpha\cdot (x,y) = (\omega_3 x,y)$, where $\alpha$ is a generator of $H$. For convenience, we fix the above notation for the remainder of the proof. 

To prove part (c), we first show that $C$ has a $k'$-rational point over some extension $k'/k$ of odd degree. Fix an element $g \in A_4 \subset \Aut(C)(k)$ of order $3$ and assume without loss of generality that $\hat{g} = \alpha$. Note that $g(q) = g(p) + \alpha(q-p)$ for any $p, q \in C(\ol{k})$. Taking $q = p^{\sigma}$ for any $\sigma \in \Gal(\ol{k}/k)$ and using the fact that $g$ is defined over $k$, we obtain that $g(p)^{\sigma}-g(p) = \alpha(p^\sigma - p)$, i.e., $(1 - \alpha)(p^\sigma - p) = P^\sigma - P$ for $P = p-g(p) \in E(\ol{k})$. By \cite[Thm. X.3.6]{Sil09}, it follows that the class $\{C/k\} \in H^1(k,E)$ belongs to the kernel of the map $(1 - \alpha)_*\colon H^1(k,E) \to H^1(k,E)$ induced by $1 - \alpha \in \mathrm{End}(E)$. However, note that $(2 + \alpha) \circ (1-\alpha) = 3$, which implies that the class $\{C/k\}$ is $3$-torsion. It follows that there exists an extension $k'/k$ such that $[k':k]$ is a power of $3$ and $C$ has a $k'$-rational point (see \cite[Prop. 5]{LT58} and the remark that follows). 

We claim that after possibly taking a cubic extension of $k'$, we can find an $A_4$-equivariant isomorphism $C \to E$. Fix a point $p_0 \in C(k')$. We would like to find $P_0 \in E(\ol{k})$ such that $(1-\alpha)(P_0) = g(p_0)-p_0 \in E(k')$. It is not hard to see that such a point $P_0$ can be found over some cubic extension of $k'$. (For example, this can be done by noting that the coordinates of $P_0$ satisfy cubic polynomials with coefficients in $k'$.) Without loss of generality, assume that $P_0 \in E(k')$ and define $q_0 = p_0+P_0 \in C(k')$. Note that $g(q_0) = g(p_0) + \alpha(P_0) = p_0 + P_0 = q_0$. We claim that the $k'$-isomorphism $\varphi\colon C \to E$ defined by $q \mapsto q-q_0$ is $A_4$-equivariant. Since it clearly commutes with translations, it suffices to show that $\varphi(g(q)) = \alpha (\varphi(q))$. We compute $\varphi(g(q)) = g(q) - q_0 = g(q) - g(q_0) = \alpha(\varphi(q))$, which completes the proof of the claim.

It remains to prove part (d). We reduce the problem to curves of genus $1$ with $k$-rational points. Assume the result is true in this case. Then we must have $\widehat{\Delta}(E) = [(-1,-1)_2]$ in $\Br_2(k(E)^{A_4})$, where $E$ is the Jacobian of $C$. By part (c), we can find an odd degree extension $k'/k$ such that $E_{k'} \cong C_{k'}$ as $A_4$-varieties. Therefore, we must have $\widehat{\Delta}(C_{k'}) = [(-1,-1)_2]$ in $\Br_2(k'(C)^{A_4})$. The natural map $\Br_2(k(C)^{A_4}) \to \Br_2(k'(C)^{A_4})$ is injective since $[k'(C)^{A_4}:k(C)^{A_4}]$ is odd, so it follows that $\widehat{\Delta}(C) = [(-1,-1)_2]$ in $\Br_2(k(C)^{A_4})$. This implies that it suffices to prove the statement for $E$.

We explicitly compute $\widehat{\Delta}(E) \in \Br(k(E)^{A_4})$. It is easy to check that the rational map $E \dashrightarrow \PP^1$ given by
$$
(x,y) \mapsto t = \frac{(y^4 + 18 b y^2 - 27b^2)}{y^3}
$$
is an $A_4$-invariant map of degree $12$, so it coincides with the rational quotient map $E \dashrightarrow E/A_4$. We may view the element $[E] \in H^1(k(t),A_4)$ as the $A_4$-Galois field extension $k(E)/k(t)$. Therefore, its corresponding {\' e}tale $k(t)$-algebra is (isomorphic to) the fixed field $k(E)^{H} = k(y)$ (recall that $A_4 = V \rtimes H$).  Note that $y$ is a root of 
$$
p(Y) = Y^4 - t Y^3 + 18 b Y^2 - 27 b^2,
$$  
so it follows that $k(y) = k(t)[Y]/(p(Y))$. A computation shows that the trace form of this {\' e}tale algebra is isomorphic to $\langle 1, A, B, AB\rangle$, where $A = 3 t^2 - 144 b$ and $B = (192 b - 3 t^2)(144 b - 3 t^2)$. It follows that its Stiefel-Whitney class is equal to $[(-A,-B)_2] + [(-1,-1)_2]$. By Proposition \ref{Prop_Computation_Invar_Polyhedral}, it suffices to show that $(-A,-B)_2$ is split over $k(t)$. Note that we have an isomorphism
$$
(-A,-B)_2 \cong (144 b - 3 t^2, 192 b - 3 t^2)_2.
$$
Recall that $-3$ is a square in $k$ because $k$ contains a primitive third root of unity. Hence the identity
$$
(144 b - 3 t^2) 2^2 + (192 b - 3 t^2) (\sqrt{-3})^2 = (\sqrt{-3}\: t)^2
$$
holds over $k(t)$, which proves that the above quaternion algebra is split.
\end{proof}

\subsection{Strong incompressibility}

\begin{prop} \label{Prop_Incompressibility_Polyhedral}
Let $G$ be a polyhedral group. The following are equivalent:
\begin{enumerate} [\upshape(i)]
\item There are no strongly incompressible $G$-curves defined over $k$.
\item $\mathrm{cd}_2(k)=0$.
\end{enumerate}
\end{prop}

\begin{proof}[Proof of (ii) $\Rightarrow$ (i)]
Suppose that $\mathrm{cd}_2(k)=0$. By Lemma \ref{Lemma_Fields_cd2}, it follows that $k$ satisfies the hypotheses of Lemma \ref{Lemma_Comp_Polyhedral_P1} below. In particular, there exists an embedding $\rho\colon G \hookrightarrow \PGL_2$ defined over $k$. We claim that any faithful $G$-curve $X$ can be $G$-compressed to $\fourIdx{}{\rho}{}{}{\PP^1}$. Indeed, the field $K = k(X)^G$ satisfies $\mathrm{cd}_2(K) = 1$ and therefore, $\Br_2(K) = 1$. Hence $\Delta_\rho(X) = 1$ and the claim follows from Corollary \ref{Cor_Maps_P1}. To finish the proof, we must show that $\fourIdx{}{\rho}{}{}{\PP^1}$ is not strongly incompressible. This is achieved in Lemma \ref{Lemma_Comp_Polyhedral_P1}.
\end{proof}

\begin{lemma} \label{Lemma_Comp_Polyhedral_P1}
Let $G$ be a polyhedral group. Suppose that $\omega_4 \in k$ if $G = A_4$ or $S_4$ (resp. $\omega_5 \in k$ if $G = A_5$), and let $\rho\colon G \hookrightarrow \PGL_2$ be an embedding defined over $k$ (it is unique up to conjugacy). Then the $G$-variety $\fourIdx{}{\rho}{}{}{\PP^1}$ is not strongly incompressible.
\end{lemma}

\begin{proof}
As the group $A_4$ is contained in $S_4$, it suffices to find non-birational compressions for $S_4$ and $A_5$. 

Case 1: Suppose that $G = S_4$. The matrices
$$
\left(\begin{array}{cc}
\omega_4 & 0\\
0 & 1
\end{array}\right), \quad
\left(\begin{array}{cc}
\omega_4 & \omega_4\\
-1 & 1
\end{array}\right),
$$
generate a subgroup isomorphic to $S_4$ inside $\PGL_2(k)$, whence we may assume that $\rho(G)$ is this particular subgroup. Then an easy computation shows that
$$
(x:y) \mapsto (7 x^4 y^3 + y^7: -x^7 - 7 x^3 y^4)
$$
is a $G$-compression $\fourIdx{}{\rho}{}{}{\PP^1} \to \fourIdx{}{\rho}{}{}{\PP^1}$, which is clearly not birational.

Case 2: Suppose that $G = A_5$. Consider the matrices
$$
\left(\begin{array}{cc}
\omega_5 & 0\\
0 & 1
\end{array}\right), \quad
\left(\begin{array}{cc}
\omega_5+\omega_5^{-1} & 1\\
1 & -\omega_5-\omega_5^{-1}
\end{array}\right);
$$
they generate a subgroup isomorphic to $A_5$ inside $\PGL_2(k)$. Again, assume that $\rho(G)$ coincides with this subgroup. Then the morphism
$$
(x:y) \mapsto (x^{11} + 66 x^6 y^5 - 11 x y^{10} : -11 x^{10} y - 66 x^5 y^6 + y^{11})
$$
is a non-birational $G$-compression $\fourIdx{}{\rho}{}{}{\PP^1} \to \fourIdx{}{\rho}{}{}{\PP^1}$.
\end{proof}

It remains to prove (i) $\Rightarrow$ (ii) in Proposition \ref{Prop_Incompressibility_Polyhedral}. The following lemma will be useful in the sequel.

\begin{lemma} \label{Lemma_Etale_Algebra}
Let $k$ be a field, let $(a,b)_2$ be a quaternion algebra defined over $k$ and let $n \geq 4$ be an integer. Then we have the following properties.
\begin{enumerate} [\upshape(a)]
\item There exists an $n$-dimensional {\' e}tale $k$-algebra $E_1$ with trivial discriminant such that the Stiefel-Whitney class $w_2(q_{E_1}) = [(a,b)_2] + [(-1,-1)_2]$,
\item If $(a,b)_2 \not\cong (-1,-1)_2$, there exists an $n$-dimensional {\' e}tale $k$-algebra $E_2$ with nontrivial discriminant $d_{E_2}$ such that $w_2(q_{E_2}) = [(a,b)_2] + [(-1,-d_{E_2})_2]$.
\end{enumerate}
\end{lemma}

\begin{proof}
It suffices to prove the results for $n=4$, as adding copies of the trivial {\' e}tale algebra $k$ to $E$ does not change the discriminant of $E$, or $w_2(q_E)$. By \cite[Lemma 31.19]{GMS03}, the $k$-algebra $E[A,B] = k[X]/(X^4-2 A X^2 + B)$ is {\' e}tale when $AB(A^2-B) \neq 0$, has discriminant $64 B (A^2 - B)^2$, and its trace form is isomorphic to $\langle 1, A, A^2-B, AB(A^2-B)\rangle$. An easy computation shows that 
$$
w_2(q_{E[A,B]}) = [(-A, -B(A^2-B))_2] + [(-1,-B)_2].
$$ 

To prove part (a), select $c \in k^{\times}$ such that $b^2 c^4 - 1 \neq 0$, and put $A = -a(b c^2 - 1)^2$ and $B = a^2(b^2 c^4-1)^2$. It is easy to see that $-A \equiv a \mod k^{\times 2}$ and $-B(A^2-B) \equiv b \mod k^{\times 2}$, whence $E_1 = E[A,B]$ satisfies the required properties.

To prove part (b), we may assume without loss of generality that $-b \not\in k^{\times 2}$ and $b \neq 1$, by changing the presentation of $(a,b)_2$ if necessary. Define $A = -a$ and $B = -4 b a^2/(b-1)^2$; then we obtain that $A^2 - B \in k^{\times 2}$. The algebra $E_2 = E[A,B]$ has discriminant $-b \not\in k^{\times 2}$ and satisfies $w_2(q_{E_2}) = [(a,b)_2] + [(-1,b)_2]$.  
\end{proof}

\begin{remark}
The conclusion in part (b) of the above theorem might fail if $(a,b)_2 \cong (-1,-1)_2$. Indeed, suppose that $k = \mathbb{R}$. By \cite[Thm. 31.18]{GMS03}, we observe that the trace form of any $4$-dimensional {\' e}tale algebra $E$ has the form $q_E = \langle 1, A, A^2-B, A B (A^2-B) \rangle$, which has second Stiefel-Whitney invariant $w_2(q_E) = [(-A, -B(A^2-B))_2] + [(-1,-B)_2]$. Since we want the discriminant to be nontrivial, $B$ must be negative, so $w_2(q_E) = [(-A, A^2-B)_2]$. This class is obviously trivial because $A^2-B > 0$, so we cannot obtain $[(-1,-1)_2]$.
\end{remark}

\begin{proof}[Proof of (i) $\Rightarrow (ii)$ in Proposition \ref{Prop_Incompressibility_Polyhedral}] Suppose that $\mathrm{cd}_2(k)>0$ and let $K = k(x)$. Note in particular that the field $K$ is Hilbertian (see \cite[Prop. 13.2.1]{FJ08}).

Case 1: Suppose that $G = A_n$, where $n = 4$ or $5$. By Lemma \ref{Lemma_Fields_cd2}, there exists a nonsplit quaternion algebra $A$ defined over $K$. Using Lemma \ref{Lemma_Etale_Algebra}(a), we can construct an $n$-dimensional {\' e}tale $K$-algebra $E$ with trivial discriminant such that $w_2(q_E) = [A] +[(-1,-1)_2]$. By \cite[Thm. 1]{EK94}, there exists a field extension $L/K$ of degree $n$ whose trace form is isometric to $q_E$; moreover, we may assume that its Galois closure $\widetilde{L}/K$ has Galois group $G$. Therefore, the class of $L$ (viewed as an {\' e}tale $K$-algebra) in $H^1(K,G)$ corresponds to a faithful $G$-curve $X$ defined over $k$ with function field $\widetilde{L}$. By Proposition \ref{Prop_Computation_Invar_Polyhedral}, we must have $\widehat{\Delta}(X) = [A]+[(-1,-1)_2]$.

We claim that $X$ cannot be $G$-compressed to any curve of genus $\leq 1$. Any faithful $G$-curve of genus $0$ is $G$-equivariantly isomorphic to $X_0$. Suppose that there exists a $G$-compression $X \to X_0$. By Proposition \ref{Prop_Invariant_Polyhedral}, the image of $\widehat{\Delta}(X_0)$ in $\Br_2(K)$ under the induced map is equal to $\widehat{\Delta}(X) = [A] + [(-1,-1)_2]$. By Proposition \ref{Prop_Invariant_Polyhedral_0}(a), it follows that $[A]$ is trivial, which is a contradiction. 

If $G = A_5$, the claim follows because $A_5$ does not act on any curve of genus $1$. On the other hand, suppose that there exists a $A_4$-compression $X \to C$, where $C$ has genus $1$. (A word of warning: Here we cannot assure that $C$ is an elliptic curve because it might not have $k$-rational points.) As before, it follows that $\widehat{\Delta}(C)$ maps to $\widehat{\Delta}(X) \in \Br_2(K)$ under the map induced by the compression. However, Lemma \ref{Prop_Genus_1_Alternating}(d) contradicts the fact that $A$ is not split. This completes the proof of the claim.                                                                                                                                                                                                                                                                                                                                                                                                                                                                                                                                                                                                                                                                                                                                                                                                                                                                                                                                                                                                                                                                                                                                                                                                                                                                                                                                                                                                                                                                                                                                                                                                                                                                                                                                                                                                                                                                                                                                                                                                                                                                                                                                                                                                                                                        By Lemma \ref{Lemma_NoCompression_Small_Genus}, there exist strongly incompressible $G$-curves.

Case 2: Suppose that $G = S_4$. We claim that there exists a quaternion algebra $A \not\cong (-1,-1)_2$ over $K$ which does not split over $k'(x)$, where $k' = k(\sqrt{2})$. If $2$ is a square in $k$ and $(-1,-1)_2$ is split over $K$, the result follows immediately from Lemma \ref{Lemma_Fields_cd2}. If $2$ is a square but $(-1,-1)_2$ is not split over $K$, we choose $A = (-1,x)_2$. Note that $A \cong (-1,-1)_2$ over $K$ if and only if $(-1,-x)_2$ is split. By Lemma \ref{Lemma_Square_Quaternions}, if either $A$ is split or $A \cong (-1,-1)_2$, it would follow that $-1$ is a square in $k$, which contradicts our assumption that $(-1,-1)_2$ is not split. 

Finally, if $2$ is not a square over $k$, we choose $A = (x, x^2 - 4 x + 2)_2$. Suppose for the sake of contradiction that $A$ splits over $k'(x)$. By Lemma \ref{Lemma_Square_Quaternions}, $2 +\sqrt{2}$ is a square over $k'$, i.e., $2 + \sqrt{2} = (l_1 + l_2 \sqrt{2})^2$ for some $l_1, l_2 \in k$. Taking norms with respect to $k'/k$ yields $2 = (l_1^2 - 2 l_2^2)^2$, which contradicts our assumption. We now prove that $A \not\cong (-1,-1)_2$, where we may assume that $(-1,-1)_2$ is not split. Indeed, such an isomorphism would imply that the quadratic forms $\langle 1, 1, 1 \rangle$
and $\langle -x, -(x^2-4x+2), x(x^2-4x+2) \rangle$ are isomorphic over $K$. It follows that $\langle 1, 1, 1 \rangle$ represents $-x$, i.e., there exist coprime polynomials $p ,q, r, s \in k[x]$ such that $p(x)^2+q(x)^2+r(x)^2 = -x s(x)^2$. Making $x = 0$ yields $p(0) = q(0) = r(0) = 0$ since we are assuming that $\langle 1, 1, 1 \rangle$ is anisotropic over $k$. This implies that $p(x) ,q(x), r(x)$ are divisible by $x$, whence $s(x)$ is divisible by $x$ as well. This contradicts the fact that $p, q, r, s$ are coprime.

By Lemma \ref{Lemma_Etale_Algebra}(b), we can construct a $4$-dimensional {\' e}tale $K$-algebra $E$ with nontrivial discriminant $d_E$ such that $w_2(q_E) = [A]+[(-1,-d_E)]$. By \cite[Thm. 1]{EK94}, we can find a field extension $L/K$ of degree $4$ whose trace form is isometric to $q_E$, whose Galois closure $\widetilde{L}/K$ has Galois group $G$. As before, its class in $H^1(K,G)$ corresponds to a faithful $G$-curve $X$ defined over $k$ with function field $\widetilde{L}$. By Proposition \ref{Prop_Computation_Invar_Polyhedral}, it follows that $\widehat{\Delta}(X) = [A] + [(-1,-1)_2] + [(2, d_E)_2]$. 

As in Case 1, suppose that there is a $G$-compression $f\colon X \to X_0$ and let $f'\colon X' \to X_0'$ be the base extension of $f$ to $k' = k(\sqrt{2})$. There exists a commutative diagram 
$$
\xymatrix{\Br_2(k(X_0)^G) \ar[r]^-{i_*} \ar[d]_-{j_0} & \Br_2(K) \ar[d]_-{j}\\
\Br_2(k'(X_0')^G) \ar[r]^-{i'_*} & \Br_2(k'(x))}
$$
where the vertical arrows are induced by base extension and the horizontal arrows are induced by $f$ and $f'$.
By Proposition \ref{Prop_Invariant_Polyhedral}, we must have $i_*(\widehat{\Delta}(X_0)) = \widehat{\Delta}(X)$ in $\Br_2(K)$. By Proposition \ref{Prop_Invariant_Polyhedral_0}(b), it follows that $j_0(\widehat{\Delta}(X_0)) = [(-1,-1)_2]$, since $2$ is a square in $k'$. Consequently, we conclude that
$$
[(-1,-1)_2] = i'_*(j_0(\widehat{\Delta}(X_0))) = j(i_*(\widehat{\Delta}(X_0))) = j(\widehat{\Delta}(X)) = [A] + [(-1,-1)_2],
$$ 
whence $A$ must be split over $k'(x)$. This contradicts our initial assumption.

Since $G$ does not act faithfully on any curve of genus $1$, it follows from Lemma \ref{Lemma_NoCompression_Small_Genus} that there exist strongly incompressible $G$-curves.
\end{proof}

\appendix
\section{Proof of Theorem \ref{Theorem_Curve_FixedPoints}} \label{Appendix_Curves}

\begin{lemma} \label{Lemma_Appendix_Polynomials}
Let $P, Q$ be two polynomials in $k[x]$, not both zero, and let $A \subset \ol{k}$ be the set of their common roots. Then for all but finitely many $c \in k$, the polynomial $P + c\, Q$ has no multiple roots outside of $A$. 
\end{lemma}

\begin{proof}
It suffices to show that given two coprime polynomials $P, Q \in k[x]$, the polynomial $P + c\, Q$ has simple roots for all but finitely many $c \in k$. If both polynomials are constant, the result is immediate, so we may assume that is not the case. Note that $\xi \in \ol{k}$ is a multiple root of $P + c\, Q$ if and only if $P(\xi) + c\, Q(\xi) = P'(\xi) + c\, Q'(\xi) = 0$, which implies that $P(\xi) Q'(\xi) - P'(\xi) Q(\xi) = 0$. The polynomial $P Q' - P' Q$ cannot be identically zero because $P$ and $Q$ are coprime and not both constant, so it has finitely many roots. If we take $c \in k$ outside the finite set
$$
\{-P(\xi)/Q(\xi) \:|\: \xi \in \ol{k} \textrm{ satisfies } P(\xi) Q'(\xi) - P'(\xi) Q(\xi) = 0,\ Q(\xi) \neq 0 \},
$$ 
it follows that $P + c\, Q$ has simple roots. The proof is complete.
\end{proof}

\begin{defin}
We define a {\em ramification condition} to be an $l$-tuple of integers $\PPP = (b_1, \ldots, b_l)$, where $l \geq 1$ and $b_i \geq 2$ for all $i$. We say that $P \in k[x]$ has a {\em local decomposition of type $\PPP$ at $\beta \in k$}, if there exists a factorization
$$
P(x) - \beta = a (x-\alpha_1)^{b_1}\ldots (x-\alpha_l)^{b_l} (x-\alpha_{l+1}) \ldots (x-\alpha_r),
$$
where $a$ is the leading coefficient of $P$, and $\alpha_1, \ldots, \alpha_r$ are distinct elements in $\ol{k}$.
\end{defin}

\begin{prop} \label{Prop_Polynomial_Ramification}
Let $\PPP_i = (b_{i,1}, \ldots, b_{i, l_i})$ ($1 \leq i \leq n$) be a collection of ramification conditions (not necessarily distinct), and let $\beta_1, \ldots, \beta_n$ be distinct points in $k$. Then there exists a polynomial $P \in k[x]$ that satisfies local decompositions of type $\PPP_i$ at $\beta_i$ for $1 \leq i \leq n$. Moreover, we can choose $\deg(P)$ to be any sufficiently large positive integer.  
\end{prop}

\begin{proof}
Choose distinct points $a_{ij} \in k$ for $1 \leq i \leq n$, $1 \leq j \leq l_i$. By the Chinese Remainder Theorem, there exists $Q \in k[x]$ such that
$$
Q(x) \equiv \beta_i + (x-a_{ij})^{b_{ij}} \!\!\! \mod  (x-a_{ij})^{b_{ij}+1},
$$
for $1 \leq i \leq n$, $1 \leq j \leq l_i$. We define
$$
H(x) = \prod_{i,j} (x-a_{ij})^{b_{ij}+1}
$$
and we let $A = \{a_{ij}\}_{i,j}$ be the set of its roots. Applying Lemma \ref{Lemma_Appendix_Polynomials} to $g_i = Q - \beta_i$ and $H$ for $1 \leq i \leq n$, we conclude that there exists a finite set $S \subset k$ such that if $c \in k$ lies outside $S$, the polynomials $g_i + c\, H$ contain no multiple roots outside of $A$ for $1 \leq i \leq n$. Choose any such $c$ and define $P = Q + c\, H$. We claim that $P$ satisfies the desired conditions. Indeed, note that the following properties hold.
\begin{enumerate}[\upshape(i)]
\item For $1 \leq i \leq n$, $1 \leq j \leq l_i$, the polynomial $P - \beta_i$ has a root of multiplicity $b_{ij}$ at the point $a_{ij}$. 
\item If $i' \neq i$, we have $P(a_{i'j}) = \beta_{i'} \neq \beta_i$ and therefore $P - \beta_i$ cannot have any root of the form $a_{i'j}$.
\item By construction, $P - \beta_i$ does not have multiple roots outside of $A$. 
\end{enumerate}

It remains to prove that we can take $\deg(P)$ to be any sufficiently large positive integer $d$. To show this, take $n = \max(\deg(Q), \deg(H))$. We claim that there exists $P$ satisfying the desired properties such that $\deg(P) = d$ for any $d > n$. Indeed, if we replace $H(x)$ by $(x-a_{11})^{d-\deg{H}} H(x)$ and ensure that $c \neq 0$ in the definition of $P$, it follows easily that $\deg(P) = d$.
\end{proof}

\begin{proof}[Proof of Theorem \ref{Theorem_Curve_FixedPoints}]

Without loss of generality, we may assume that $G = S_m$ for some $m \geq 2$. Given a partition $b_1 + \ldots + b_s$ of $m$, where $b_1 \geq \ldots \geq b_l > 1 = b_{l+1} = \ldots = b_{s}$ for some $l \geq 1$, we can define a ramification condition $\PPP = (b_1, \ldots, b_l)$. Let $\PPP_1, \ldots, \PPP_n$ be the ramification conditions obtained as we range over all possible partitions of $m$, except for $1 + \ldots + 1$. By Proposition \ref{Prop_Polynomial_Ramification}, we can construct a polynomial $P \in k[x]$ satisfying local decompositions of type $\PPP_i$ at distinct points $\beta_i$ for $1 \leq i \leq n$. Moreover, we may assume that $\deg(P)$ is some sufficiently large prime number $p$.

Let the group $S_p$ act on $p$ letters and embed $S_m$ inside $S_p$ as the subgroup that fixes the last $p-m$ letters. We want to construct $X$ as a ramified $S_p$-cover of $\PP^1$. Let $P_t(x) = P(x) - t$, where $t$ is an indeterminate, and define $L$ as the splitting field of $P_t$ over $k(t)$. It is clear that $\Gal(L/k(t))$ is a transitive subgroup of $S_p$; we claim that equality holds. Since $\Gal(L\ol{k}/\ol{k}(t))$ is a subgroup of $\Gal(L/k(t))$, it suffices to prove that the former is isomorphic to $S_p$. (Note that this also implies that $L$ is regular, i.e., $L \cap \ol{k} = k$.) We use a technique similar to \cite[Thm. 4.4.5]{Se08}. We may view the polynomial $P_t$ as a ramified cover $\PP^1 \to \PP^1$ of degree $p$. Note that $\beta_1, \ldots, \beta_n, \infty$ are among the ramification points. If $\PPP_i = (b^{(i)}_1, \ldots, b^{(i)}_{l_i})$, the inertia subgroup at $\beta_i$ is generated by an element of $S_p$ of cycle type $(b^{(i)}_1, \ldots, b^{(i)}_{l_i}, 1, \ldots, 1)$, while the inertia group at $\infty$ is a $p$-cycle. In particular, $\Gal(L\ol{k}/\ol{k}(t))$ contains the subgroup generated by a $p$-cycle and a transposition, which is all of $S_p$ since $p$ is prime. The claim follows immediately.

Let $X$ be the (unique) smooth projective curve defined over $k$ with function field $L$, which is geometrically irreducible since $L/k$ is regular. Note that $X$ can be endowed with a natural faithful $S_p$-action via the Galois action on $L$. If $Q$ is a closed point in $X_{\ol{k}}$ lying above $\beta_i$, then its stabilizer is a cyclic subgroup generated by an element of $S_p$ of cycle type $(b^{(i)}_1, \ldots, b^{(i)}_{l_i}, 1, \ldots, 1)$. Since any two subgroups of this form are conjugate, they all occur as stabilizers of points in the fibre above $\beta_i$. Clearly, any nontrivial element of $S_m$ has one of the above cycle types inside $S_p$, so it fixes at least one geometric point in $X$. The proof is complete.
\end{proof}

\end{document}